\documentclass[10pt]{amsart}

\usepackage{mathrsfs}
\usepackage{amstext}
\usepackage{amscd}
\usepackage{latexsym}
\usepackage{amsfonts}
\usepackage{amssymb}
\usepackage{amsmath}
\usepackage{amsthm}
\usepackage{color}
\usepackage{multicol}
\usepackage[pdftex]{graphicx}
\graphicspath{{./pict/}}            % Џгвм Є Ї ЇЄҐ б аЁбг­Є ¬Ё
\DeclareGraphicsExtensions{.pdf,.png,.jpg}

\usepackage[cp1251]{inputenc}
\usepackage[russian,english]{babel}
\theoremstyle{plain}
\newtheorem{theorem}{Theorem}[section]
\newtheorem{lemma}[theorem]{Lemma}
\newtheorem{prepos}[theorem]{Proposition}
\newtheorem{corol}[theorem]{Corollary}

\theoremstyle{definition}
\newtheorem{definition}[theorem]{Definition}

\newtheorem{remark}[theorem]{Remark}

\renewcommand{\Im}{\operatorname{Im}}
\renewcommand{\Re}{\operatorname{Re}}

\def\eqbd{\mathop{{:}{=}}}

\begin{document}
\title[Total nonnegativity and root location]
{Total nonnegativity of finite Hurwitz matrices and root location of polynomials}

%----------Author 1
\author[M.~Adm]{Mohammad Adm}

\address{Department of Mathematics and Statistics, University of Konstanz, Konstanz, Germany, Department of Applied Mathematics and Physics, Palestine Polytechnic University, Hebron, Palestine, and Department of Mathematics and Statistics, University of Regina, Regina, Canada.}
\email{mjamathe@yahoo.com}
%\email{moh\_95@ppu.edu}

%----------Author 2
\author[J.~Garloff]{J\"urgen Garloff}

\address{Department of Mathematics and Statistics, University of Konstanz\\
and Institute for Applied Research, University of Applied Sciences/HTWG Konstanz}
\email{garloff@htwg-konstanz.de}
\email{Juergen.Garloff@htwg-konstanz.de}

%----------Author 3
\author[M.~Tyaglov]{Mikhail Tyaglov}

\address{School of Mathematical Sciences, Shanghai Jiao Tong University}
\email{tyaglov@sjtu.edu.cn}
%\vspace{\baselineskip} \thispagestyle{empty}

%\author{The authors}%\thanks{The work was supported by the
%Sofja Kovalevskaja Research Prize of Alexander von Humboldt
%Foundation.
%Email: {\tt tyaglov@math.tu-berlin.de}}  \\
%\small Institut f\"ur Mathematik,  Technische Universit\"at
%Berlin }

%15B48 - Positive matrices and their generalizations; cones of matrices
%15B35 - Sign-pattern matrices
%15A18 - Eigenvalues, singular values, and eigenvectors
%12D10 - Polynomials: location of zeros (algebraic theorems)
%26C10 - Polynomials: location of zeros (analytic theorems);
%26C15 - Rational functions;
%30C10 - Polynomials (geometry)
%30C15 - (geometry of roots) Zeros of polynomials, rational functions, and other analytic functions (e.g. zeros of functions %with bounded Dirichlet integral)
%15B05 - Toeplitz, Cauchy, and related matrices


%\date{January 1, 2004}
%----------additions
%\dedicatory{To Yury Barkovsky}
%%% ----------------------------------------------------------------------

\begin{abstract}
In 1970, B.A.\,Asner, Jr., proved that for a real quasi-stable polynomial, i.e.,
a polynomial whose zeros lie in the \emph{closed} left half-plane of the complex plane,
its finite Hurwitz matrix is totally nonnegative, i.e., all its minors are nonnegative, and that the converse statement is not true.
In this work, we explain this phenomenon in detail, and provide necessary and sufficient conditions
for a real polynomial to have a totally nonnegative finite Hurwitz matrix.
\end{abstract}

%%% ----------------------------------------------------------------------
\maketitle
%%% ----------------------------------------------------------------------
%\tableofcontents
%\date{\small \today}

%\begin{titlepage}
%\thispagestyle{empty}

\setcounter{equation}{0}

%%% ----------------------------------------------------------------------
%\tableofcontents
%%%%%%%%%%%%%%%%%%%%%%%%%%%%%%%%%%%%%%%%%%%%%%%%%%%%%%%%
\section{Introduction}\label{section:intro}
%%%%%%%%%%%%%%%%%%%%%%%%%%%%%%%%%%%%%%%%%%%%%%%%%%%%%%%%

This paper is devoted to total nonnegativity of Hurwitz matrices. We remind the reader
that given a real polynomial of degree $n$
\begin{equation}\label{intro.main.polynomial}
p(z)=a_0z^n+a_1z^{n-1}+\dots+a_n,\qquad a_0,\dots,a_n\in\mathbb
R,\quad a_0,a_n>0,
\end{equation}
its \emph{finite} Hurwitz matrix has the form
\begin{equation}\label{HurwitzMatrix}
\mathcal{H}_n(p)=
\begin{pmatrix}
a_1&a_3&a_5&a_7&\dots&0&0\\
a_0&a_2&a_4&a_6&\dots&0&0\\
0  &a_1&a_3&a_5&\dots&0&0\\
0  &a_0&a_2&a_4&\dots&0&0\\
\vdots&\vdots&\vdots&\vdots&\ddots&\vdots&\vdots\\
0  &0  &0  &0  &\dots&a_{n-1} &0\\
0  &0  &0  &0  &\dots&a_{n-2} &a_n
\end{pmatrix}\,.
\end{equation}

In 1970, B.A.\,Asner, Jr., established in~\cite{Asner} that if the polynomial $p$ is quasi-stable
(that is, all its zeros lie in the \textit{closed} left half-plane of complex plane), then the
matrix $\mathcal{H}_n(p)$ is totally nonnegative. This means that all its minors are
nonnegative. Asner noted that the converse statement is not true. As an example, he provided
the polynomial $p(z)=z^4+198z^2+10201$ with zeros $\pm1\pm i10$ whose finite Hurwitz
matrix $\mathcal{H}_4(p)$ is totally nonnegative. In fact, Asner implicitly established that if the
finite Hurwitz matrix of a real polynomial is nonsingular and totally nonnegative, then
this polynomial is (Hurwitz) stable (that is, all its zeros lie in the \textit{open} left half-plane of
the complex plane). In 1980, J.\,H.\,B.\,Kemperman~\cite{Kemperman} considered the \textit{infinite}
Hurwitz matrix
\begin{equation}\label{Hurwitz.matrix.infinite.for.poly}
H_{\infty}(p)=
\begin{pmatrix}
a_0&a_2&a_4&a_6&a_8&a_{10}&\dots\\
0  &a_1&a_3&a_5&a_7&a_9&\dots\\
0  &a_0&a_2&a_4&a_6&a_8&\dots\\
0  &0  &a_1&a_3&a_5&a_7&\dots\\
\vdots &\vdots  &\vdots  &\vdots &\vdots  &\vdots &\ddots
\end{pmatrix}
\end{equation}
and proved by a method different from Asner's one that the matrix~\eqref{Hurwitz.matrix.infinite.for.poly} is totally nonnegative if
the polynomial $p$ given in~\eqref{intro.main.polynomial} is quasi-stable. Later on, O.\,Holtz~\cite{Holtz1} gave a very simple proof of this fact. However, both Kemperman and Holtz did not discuss the converse statement.

In~\cite[Theorem~3.44]{Holtz_Tyaglov},  a general theorem was proved which implies that the total nonnegativity of the infinite Hurwitz matrix of a given real polynomial
is equivalent to the quasi-stability of this polynomial. This fact was also mentioned in~\cite{Dyachenko}.

To make the present work self-contained, we  mention some important properties of stable and quasi-stable polynomials which we use to obtain our main result on the total nonnegativity of finite Hurwitz matrices and prove that the total nonnegativity of the infinite Hurwitz matrix of a polynomial is equivalent to the quasi-stability
of this polynomial. Note that Asner and Kemperman initially proved their theorems for stable polynomials and extended the results to quasi-stable polynomials by approximating quasi-stable polynomials by stable polynomials. Holtz dealt only with stable polynomials. Here we consider quasi-stable polynomials directly and obtain all results for stable polynomials as a particular case.

The main results of the paper provide necessary and sufficient conditions on the zeros of a given real polynomial for its \textit{finite} Hurwitz matrix to be totally nonnegative. Note that our (sharp) necessary condition does not coincide with our (sharp) sufficient condition. To obtain these conditions we use results by I.\,Schoenberg~\cite{Schoenberg.m.pos.1,Schoenberg.m.pos.2} on the polynomials from the class of the P\'olya frequency functions, see Section~\ref{section:proofs} for details.
These conditions require that the given polynomial does not have zeros in a specified sector in the right half-plane. Such polynomials appear in the stability analysis of fractional differential equations, e.g., commensurate fractional-order linear time-invariant systems~\cite{Matignon} and fractional-order Lotka-Volterra predator-prey models~\cite{Ahmed_etc}.

In passing we note some properties of the Hurwitz matrix which are stronger than its total nonnegativity. It was noted by Kemperman~\cite{Kemperman} that the infinite Hurwitz matrix associated with a stable polynomial is almost totally positive, i.e., besides its total nonnegativity, each of its square submatrices has a positive determinant if and only if all of the diagonal entries of this submatrix are positive. It was shown in~\cite{Gasca.Micchelli.Pena} that, in fact, the latter positivity condition suffices to be hold only for all square submatrices formed from consecutive rows and columns. Characterizations of the almost total positivity of the infinite matrices of Hurwitz type,
see Definition~\ref{def.Hurwitz.matrix.infinite}, can be found in~\cite{AGT}. We mention also that in~\cite{KatVishn} the smallest possible constant $c_n$ was determined such that the positivity of the coefficients of the polynomial $p$ given by~\eqref{intro.main.polynomial} and the satisfaction of the inequalities $a_k a_{k+1}> c_n a_{k+2} a_{k-1}$, $k=1,\ldots,n-2$, imply the stability of $p$.
Furthermore, it was shown in~\cite{Kleptsyn} that if $p$ has positive coefficients and satisfies the inequality
$$
\dfrac{a_0a_3}{a_1a_2}+\dfrac{a_1a_4}{a_2a_3}+\cdots+\dfrac{a_{n-3}a_n}{a_{n-2}a_{n-1}}<1,
$$
then $p$ is stable.

The organization of the paper is as follows. In Section~\ref{section:main.results}, we state our main results. We provide some auxiliary facts on \textit{R}-functions, the definition of the finite and infinite matrices of Hurwitz type, as well as their factorizations in Section~\ref{section:R.function}.
In Section~\ref{section:stable.poly}, we recall and prove some properties of stable and quasi-stable polynomials and
establish the results of Asner, Kemperman, and Holtz. In~Section~\ref{section:proofs}, we prove the main results of this work, Theorems~\ref{Theorem.nes.cond}
and~\ref{Theorem.suf.cond}. Section~\ref{section:eigenspace} is devoted to the eigenstructure of totally nonnegative finite
Hurwitz matrices. Here we generalize the results by Asner~\cite{Asner} and Lehnigk~\cite{Lehnigk} on the eigenvalues and Jordan form
of totally nonnegative Hurwitz matrices. In Section~7, we draw some conclusions and pose an open problem.

\setcounter{equation}{0}

%%%%%%%%%%%%%%%%%%%%%%%%%%%%%%%%%%%%%%%%%%%%%%%%%%%%%%%%%%%%%%%%%%%%%%%%%%%%%%%%%%%%%%%%%%%%%%%%%%%%%%%%%%%%%%%
\section{Main results}\label{section:main.results}
%%%%%%%%%%%%%%%%%%%%%%%%%%%%%%%%%%%%%%%%%%%%%%%%%%%%%%%%%%%%%%%%%%%%%%%%%%%%%%%%%%%%%%%%%%%%%%%%%%%%%%%%%%%%%%%

In this section we state our main results which are to be proved in Section~\ref{section:proofs}.

\begin{theorem}\label{Theorem.nes.cond}
Let $p$ be a polynomial of degree $n\geqslant2$ given in~\eqref{intro.main.polynomial}. If its finite Hurwitz matrix~$\mathcal{H}_n(p)$
is totally nonnegative, then $p$ has no zeros in the sector
\begin{equation}\label{sector.2}
|\arg z|<
\begin{cases}
&\dfrac{\pi}{4}\cdot\dfrac{n+1}{n-1}\quad\text{for odd}\ n,\\
& \\
&\dfrac{\pi}{4}\cdot\dfrac{n}{n-1}\quad\text{for even}\ n.
\end{cases}
\end{equation}
The constant in~\eqref{sector.2} is sharp.
\end{theorem}

For example, for even $n$ the finite Hurwitz matrix of the following polynomial
$$
p(z)=\prod_{j=1}^{\tfrac{n}{2}}\left(z^2+e^{i\tfrac{\pi}{2}\tfrac{n-4j+2}{n-1}}\right)
$$
is totally nonnegative, but $p(z)$ has zeros on the border of the sector~\eqref{sector.2}. Analogously,
for odd $n$ the polynomial
$$
p(z)=(z+1)\prod_{j=1}^{\tfrac{n-1}{2}}\left(z^2+e^{i\tfrac{\pi}{2}\tfrac{n-4j+1}{n-1}}\right)
$$
provides the sharp constant in Theorem~\ref{Theorem.nes.cond}.

Since $\frac{n+1}{n-1}>\frac{n}{n-1}>1$ for any $n\geqslant 2$, we can give a universal estimate for the sector free of zeros of $p$ which is independent on
the degree of the polynomial $p$.
\begin{corol}\label{Corol.nes.cond.pi/4}
If the finite Hurwitz matrix of a real polynomial $p$ is totally nonnegative, then $p$ has no zeros in the sector
\begin{equation*}
|\arg z|\leqslant\dfrac{\pi}{4}.
\end{equation*}
\end{corol}

The next theorem provides a sharp sufficient condition on a real polynomial to have its finite Hurwitz matrix totally nonnegative.

\begin{theorem}\label{Theorem.suf.cond}
Let a polynomial $p$ of degree $n\geqslant4$ given in~\eqref{intro.main.polynomial} have no zeros in the sector
\begin{equation}\label{sector.4}
|\arg z|<\dfrac{\pi}{2}\cdot\dfrac{n-2}{n-1},
\end{equation}
and satisfy the ''reflection property'': if $p(\lambda)=0$ for some $\lambda\in\mathbb{C}$ such that $\Re\lambda>0$, then  $p(-\lambda)=0$.

Then the finite Hurwitz matrix~$\mathcal{H}_n(p)$ of the polynomial $p$ is totally nonnegative.
The constant in~\eqref{sector.4} is sharp.

The polynomial $p$ of degree $n$, $1\leqslant n\leqslant3$, is quasi-stable if and only if $\mathcal{H}_n(p)$ is totally nonnegative.
\end{theorem}

\noindent The polynomial
$$
p(z)=(z+1)^{n-4}(z^4+2z^2\cos\theta+1), \quad \theta=\dfrac{\pi}{2(n-1)}+\varepsilon,
$$
with $\varepsilon>0$ arbitrarily small, has a root inside the sector~\eqref{sector.4}, and its finite Hurwitz matrix $\mathcal{H}_n(p)$ is not totally nonnegative.
This means that we cannot decrease the angle in~\eqref{sector.4}, so the result of Theorem~\ref{Theorem.suf.cond} is sharp.

\setcounter{equation}{0}
%%%%%%%%%%%%%%%%%%%%%%%%%%%%%%%%%%%%%%%%%%%%%%%%%%%%%%%%%%%%%%%%%%%%%%%%%%%%%%%%%%%%%%%%%%%%%%%%%%%%%%%%%%%
\section{Auxiliary facts: \textit{R}-functions and finite Hurwitz matrix factorization}\label{section:R.function}
%%%%%%%%%%%%%%%%%%%%%%%%%%%%%%%%%%%%%%%%%%%%%%%%%%%%%%%%%%%%%%%%%%%%%%%%%%%%%%%%%%%%%%%%%%%%%%%%%%%%%%%%%%%

Many properties of Hurwitz matrices and stable polynomials are related to properties of the so-called rational \textit{R}-functions~\cite{Gantmakher,Barkovsky.2,Holtz_Tyaglov}.

\hspace{4mm} Consider a rational function
\begin{equation}\label{basic.rational.function}
R(z)=\dfrac{q(z)}{p(z)},
\end{equation}
where $p$ and $q$ are real polynomials
 \begin{eqnarray}\label{polynomial.1}
& p(z)= a_0z^n+a_1z^{n-1}+\dots+a_n,\qquad &
a_i\in\mathbb{R},\ i=0,1,\ldots,n,\ a_0>0, \\\label{polynomial.2}
& q(z)=
b_0z^{n}+b_1z^{n-1}+\dots+b_{n},\qquad & b_i\in\mathbb{R},\ i=0,1,\ldots,n,
\end{eqnarray}
so that $\deg p=n$ and $\deg q\leqslant n$. If the greatest common
divisor of $p$ and $q$ has degree $l$, then the
rational function $R$ has exactly $r=n-l$ poles.

\begin{definition}\label{def.S-function}
A rational function $R$ is called \textit{\textit{R}-function} if it maps
the upper half-plane of the complex plane to the lower
half-plane\footnote{In~\cite{Holtz_Tyaglov} such functions are
called \textit{R}-functions of \textit{negative type}.}:
\begin{equation}\label{R-function.negative.type.condition.doblicate}
\Im z>0\Rightarrow\Im R(z)<0.
\end{equation}
\end{definition}
By now, these functions, as well as their meromorphic analogues,
have been considered by many authors and have acquired various
names. For instance, these functions are called \emph{strongly
real functions} in the monograph~\cite{Sheil-Small} due to their
property to take real values only for real values of the
argument (a more general and detailed discussion can be found
in~\cite{ChebotarevMeiman}, see also~\cite{Holtz_Tyaglov}).

Let us associate to the rational function~\eqref{basic.rational.function}--\eqref{polynomial.2} the following matrix:

\noindent If $\deg q < \deg p $, that is, if $b_0=0$, then
\begin{equation}\label{Hurwitz.matrix.infinite.case.1}
H(p,q)=
\begin{pmatrix}
a_0&a_1&a_2&a_3&a_4&a_5&\dots\\
0  &b_1&b_2&b_3&b_4&b_5&\dots\\
0  &a_0&a_1&a_2&a_3&a_4&\dots\\
0  &0  &b_1&b_2&b_3&b_4&\dots\\
\vdots &\vdots  &\vdots  &\vdots &\vdots  &\vdots &\ddots
\end{pmatrix};
\end{equation}
if $\deg q =\deg p$, that is, $b_0\neq0$, then
\begin{equation}\label{Hurwitz.matrix.infinite.case.2}
H(p,q)=
\begin{pmatrix}
b_0&b_1&b_2&b_3&b_4&b_5&\dots\\
0  &a_0&a_1&a_2&a_3&a_4&\dots\\
0  &b_0&b_1&b_2&b_3&b_4&\dots\\
0  &0  &a_0&a_1&a_2&a_3&\dots\\
\vdots &\vdots  &\vdots  &\vdots &\vdots  &\vdots &\ddots
\end{pmatrix}.
\end{equation}
\begin{definition}\label{def.Hurwitz.matrix.infinite}
The matrix $H(p,q)$ is called the \textit{infinite matrix of
Hurwitz type}. We denote its leading principal minor
of order $j$, $j=1,2,\ldots$, by $\eta_{j}(p,q)$.
\end{definition}

In~\cite{Holtz_Tyaglov} it was noticed that if $g=\gcd(p,q)$, then $\deg g=l$ if and only if the following
holds
\begin{equation}\label{degeneracy.R.func}
\eta_{n-l}(p,q)\neq0\quad\text{and}\quad \eta_{j}(p,q)=0,\quad j>n-l.
\end{equation}
In this case, the matrix $H(p,q)$ can be factorized as follows~\cite{Holtz_Tyaglov}.
\begin{theorem}[\cite{Holtz_Tyaglov}]\label{Th.Hurwitz.matrix.with.gcd}
If $g(z)=g_0z^l+g_1z^{l-1}+\cdots+g_l$, then
\begin{equation}\label{Hurwitz.matrix.with.gcd}
H(p\cdot g,q\cdot g)=H(p,q)\mathcal{T}(g),
\end{equation}
where $\mathcal{T}(g)$ is the infinite upper triangular Toeplitz
matrix formed from the coefficients of the polynomial~$g$:
\begin{equation}\label{Toeplitz.infinite.matrix}
\mathcal{T}(g)=
\begin{pmatrix}
g_0&g_1&g_2&g_3&g_4&\dots\\
0  &g_0&g_1&g_2&g_3&\dots\\
0  &0  &g_0&g_1&g_2&\dots\\
0  &0  &0  &g_0&g_1&\dots\\
0  &0  &0  &0  &g_0&\dots\\
\vdots &\vdots  &\vdots  &\vdots  &\vdots &\ddots
\end{pmatrix}.
\end{equation}
Here we set $g_i \eqbd 0$ for all $i>l$.
\end{theorem}

Moreover, the following two theorems on properties of \textit{R}-functions were established in~\cite{Holtz_Tyaglov}.

\begin{theorem}[\cite{Holtz_Tyaglov}]\label{Theorem.R-functions.with.positive.poles.via.infinite.Hurwitz.matrix}
The
function~\eqref{basic.rational.function} is an \textit{R}-function of negative type with exactly
$k$ poles, all of which are negative, if and only
if
\begin{eqnarray}\label{R-function.positive poles.criterion.via.infinite.Hurwitz.matrix.condition.1}
& \eta_{j}(p,q)>0, &  \quad j=1,2,\ldots,k,  \\
\label{R-function.positive poles.criterion.via.infinite.Hurwitz.matrix.condition.2}
& \eta_{j}(p,q)=0, &  \quad j>k,
\end{eqnarray}
where $k=2r+1$ if $\deg q<\deg p$, $k=2r+2$ if $\deg q=\deg p$, and $\eta_{j}(p,q)$ is the $j\times j$ leading
principal minor of the matrix $H(p,q)$ defined in~\eqref{Hurwitz.matrix.infinite.case.1}--\eqref{Hurwitz.matrix.infinite.case.2}.
\end{theorem}

\begin{theorem}[\textbf{total nonnegativity of the infinite Hurwitz matrix}, \cite{Holtz_Tyaglov}]\label{Th.Hurwitz.Matrix.Total.Nonnegativity}
The following statements are equivalent:
\begin{itemize}
\item[$1)$] The polynomials $p$ and $q$ defined by~\eqref{polynomial.1}--\eqref{polynomial.2} have only nonpositive
zeros\footnote{Here we include the case when $q(z)\equiv0$.}, and
the function $R=q/p$ is either an \textit{R}-function of negative
type or identically zero.
\item[$2)$] The infinite matrix of Hurwitz type $H(p,q)$ defined by~\eqref{Hurwitz.matrix.infinite.case.1}--\eqref{Hurwitz.matrix.infinite.case.2} is totally nonnegative.
\end{itemize}
\end{theorem}
Thus, the inequalities~\eqref{R-function.positive poles.criterion.via.infinite.Hurwitz.matrix.condition.1} and
equalities~\eqref{R-function.positive poles.criterion.via.infinite.Hurwitz.matrix.condition.2} are constitute a necessary and
sufficient condition for total nonnegativity of the matrix $H(p,q)$. We use these facts to describe some properties
of stable and quasi-stable polynomials.

Finally, we remind of a remarkable result established in a more general form in~\cite{AisenEdreiShoenbergWhitney,ASW}
(see also~\cite{Schoenberg.m.pos.2,Karlin}).

\begin{theorem}\label{Th.Shoenberg}
The polynomial
$$
g(z)=g_0z^l+g_1z^{l-1}+\cdots+g_l,\quad g_0g_l\neq0,
$$
has only negative zeros if and only if its Toeplitz matrix
$\mathcal{T}(g)$ defined by~\eqref{Toeplitz.infinite.matrix} is
totally nonnegative.
\end{theorem}

\vspace{2mm}

Together with the infinite matrix $H(p,q)$, we consider its finite submatrices:
\begin{definition}\label{def.Hurwitz.matrix.finite} Let the polynomials
$p$ and $q$ be given by~\eqref{polynomial.1}--\eqref{polynomial.2}.
If $\deg q < \deg p = n$, let $\mathcal{H}_{2n}(p,q)$ denote
the following ${2n} \times {2n}$-matrix:
\begin{equation}\label{Hurwitz.matrix.finite.case.1}
\mathcal{H}_{2n}(p,q)=
\begin{pmatrix}
    b_1 &b_2 &b_3 &\dots &b_{n}   &   0    &0      &\dots &0&0\\
    a_0 &a_1 &a_2 &\dots &a_{n-1} & a_{n}  &0      &\dots &0&0\\
     0  &b_1 &b_2 &\dots &b_{n-1} & b_{n}  &0      &\dots &0&0\\
     0  &a_0 &a_1 &\dots &a_{n-2} & a_{n-1}&a_{n}  &\dots &0&0\\
    \vdots&\vdots&\vdots&\ddots&\vdots&\vdots&\vdots&\ddots&\vdots&\vdots\\
     0  &  0 &  0 &\dots &a_{0} & a_{1}&a_{2}&\dots &a_{n}&0\\
     0  &  0 &  0 &\dots &   0  & b_{1}&b_{2}&\dots &b_{n}&0\\
     0  &  0 &  0 &\dots &0     & a_{0}&a_{1}&\dots &a_{n-1}&a_{n}\\
\end{pmatrix}.
\end{equation}
If $\deg q = \deg p =n$, let $\mathcal{H}_{2n+1}(p,q)$ denote the following
 $(2n{+}1)\times (2n{+}1)$-matrix
\begin{equation}\label{Hurwitz.matrix.finite.case.2}
\mathcal{H}_{2n+1}(p,q)=
\begin{pmatrix}
    a_0 &a_1 &a_2 &\dots &a_{n-1} & a_{n}  &0&\dots &0&0\\
    b_0 &b_1 &b_2 &\dots &b_{n-1} & b_{n}  &0&\dots &0&0\\
     0  &a_0 &a_1 &\dots &a_{n-2} & a_{n-1}&a_{n}&\dots &0&0\\
     0  &b_0 &b_1 &\dots &b_{n-2} & b_{n-1}&b_{n}&\dots &0&0\\
    \vdots&\vdots&\vdots&\ddots&\vdots&\vdots&\vdots&\ddots&\vdots&\vdots\\
     0  &  0 &  0 &\dots &a_{0} & a_{1}&a_{2}&\dots &a_{n}&0\\
     0  &  0 &  0 &\dots &b_{0} & b_{1}&b_{2}&\dots &b_{n}&0\\
     0  &  0 &  0 &\dots &0     & a_{0}&a_{1}&\dots &a_{n-1}&a_{n}\\
\end{pmatrix}   .
\end{equation}
Both matrices $\mathcal{H}_{2n}(p,q)$ and $\mathcal{H}_{2n+1}(p,q)$  are
called \textit{finite matrices of Hurwitz type}. The leading principal minors of
these matrices are denoted by\footnote{That is, $\Delta_j(p,q)$ is the leading
principal minor  of the matrix $\mathcal{H}_{2n}(p,q)$ of order
$j$ if $\deg q < \deg p $. Otherwise (when $\deg q =\deg p$),
 $\Delta_j(p,q)$ denotes the  leading principal minor of the
matrix $\mathcal{H}_{2n+1}(p,q)$ of order $j$.} $\Delta_j(p,q)$.
\end{definition}

Analogously to~\eqref{Hurwitz.matrix.with.gcd}, one can factorize finite Hurwitz matrices.
\begin{theorem}\label{Thm.fin.Hurw.matr.factor.1}
If $\deg p=\deg q+1=n$ and $\deg g=m$, then
\begin{equation}\label{finite.Hurwitz.matrix.with.gcd.factor.1}
\mathcal{H}_{2n+2m}(p\cdot g,q\cdot g)=\mathcal{H}_{2n+2m}(p,q)\mathcal{T}_{2n+2m}(g),
\end{equation}
where $\mathcal{H}_{2n+2m}(p,q)$ is the principal submatrix of $H(p,q)$ of order $2n+2m$ indexed by rows (and columns)
$2$ through $2n+2m+1$, and the matrix $\mathcal{T}_{2n+2m}(g)$ is the leading principal submatrix of the matrix $\mathcal{T}(g)$ of order $2n+2m$.

Moreover, if $\det\left[\mathcal{H}_{2n}(p,q)\right]\neq0$ and $p(0)\neq0$, then rank of $\mathcal{H}_{2n+2m}(p\cdot g,q\cdot g)$ equals $2n+m$.
\end{theorem}
\begin{proof}
Multiplication of the matrices $\mathcal{H}_{2n+2m}(p,q)$ and $\mathcal{T}_{2n+2m}(g)$ shows that the factorization is true.

By the Cauchy-Binet formula, rank of the matrix $\mathcal{H}_{2n+2m}(p\cdot g,q\cdot g)$ equals to rank of the matrix $\mathcal{H}_{2n+2m}(p,q)$,
since the matrix $\mathcal{T}_{2n+2m}(g)$ is nonsingular as a triangular matrix with nonzero diagonal. At the same time, $\mathcal{H}_{2n+2m}(p,q)$ has
$m$ zero columns, so its rank is at most $2n+m$. However, if $p(0)\neq0$ and $\det\left[\mathcal{H}_{2n}(p,q)\right]\neq0$, then
the determinant of $\mathcal{H}_{2n+2m}(p,q)$ of order $2n+m$ formed with the columns $1$,~$2$,~\ldots,~$2n+m$, and with the rows $1$, $2$,\ldots, $2n$, $2n+2$, $2n+4$, \ldots, $2n+2m$, equals $\det\left[\mathcal{H}_{2n}(p,q)\right]\cdot[p(0)]^m$ which is nonzero.
\end{proof}
In the same way as above, one can establish the following fact.
\begin{theorem}\label{Thm.fin.Hurw.matr.factor.2}
If $\deg p=\deg q=n$ and $\deg g=m$, then
\begin{equation}\label{finite.Hurwitz.matrix.with.gcd.factor.2}
\mathcal{H}_{2n+2m+1}(p\cdot g,q\cdot g)=\mathcal{H}_{2n+2m+1}(p,q)\mathcal{T}_{2n+2m+1}(g),
\end{equation}
where $\mathcal{H}_{2n+2m+1}(p,q)$ is the principal submatrix of $H(p,q)$ of order $2n+2m+1$ indexed by rows (and columns)
$2$ through $2n+2m+2$, and the matrix $\mathcal{T}_{2n+2m+1}(g)$ is the leading principal submatrix of the matrix $\mathcal{T}(g)$ of order $2n+2m+1$.

Moreover, if $\det\left[\mathcal{H}_{2n+1}(p,q)\right]\neq0$ and $p(0)\neq0$, then  rank of $\mathcal{H}_{2n+2m+1}(p\cdot g,q\cdot g)$ equals $2n+m+1$.
\end{theorem}

Note that the factorizations~\eqref{finite.Hurwitz.matrix.with.gcd.factor.1}--\eqref{finite.Hurwitz.matrix.with.gcd.factor.2} are simply extended
versions of the factorization~$(3.70)$ in~\cite{Holtz_Tyaglov}.

\setcounter{equation}{0}
%%%%%%%%%%%%%%%%%%%%%%%%%%%%%%%%%%%%%%%%%%%%%%%%%%%%%%%%%%%%%%%%%%%%%%%%%%%%%%%%%%%%%%%%%%%%%%%%%%%%%%%%%%
\section{Quasi-stable polynomials and total nonnegativity of Hurwitz matrices}\label{section:stable.poly}
%%%%%%%%%%%%%%%%%%%%%%%%%%%%%%%%%%%%%%%%%%%%%%%%%%%%%%%%%%%%%%%%%%%%%%%%%%%%%%%%%%%%%%%%%%%%%%%%%%%%%%%%%%

%\hspace{4mm} This section is devoted to some basic facts of the
%theory of Hurwitz stable and quasi-stable polynomials related to
%total nonnegativity of their finite and infinite Hurwitz matrices.

Consider a real polynomial
\begin{equation}\label{main.poly}
p(z)=a_0z^n+a_1z^{n-1}+\cdots+a_{n-1}z+a_n,\qquad a_0>0,\ \ a_n\neq0.
\end{equation}

Throughout this section we use the following notation
\begin{equation}\label{floor.poly.degree}
l=\left[\dfrac n2\right],
\end{equation}
where $n=\deg p$, and $[\rho]$ denotes the largest integer not
exceeding $\rho$.

The polynomial $p$ can always be represented as follows
\begin{equation}\label{app.poly.odd.even}
p(z)=p_0(z^2)+zp_1(z^2),
\end{equation}
where

for $n=2l$,
\begin{equation}\label{poly1.13}
\begin{split}
&p_0(u)=a_0u^l+a_2u^{l-1}+\ldots+a_{n},\\
&p_1(u)=a_1u^{l-1}+a_3u^{l-2}+\ldots+a_{n-1},
\end{split}
\end{equation}

and for $n=2l+1$,
\begin{equation}\label{poly1.12}
\begin{split}
&p_0(u)=a_1u^l+a_3u^{l-1}+\ldots+a_{n},\\
&p_1(u)=a_0u^l+a_2u^{l-1}+\ldots+a_{n-1}.
\end{split}
\end{equation}

We introduce the following function\footnote{In the
book~\cite[Chapter XV]{Gantmakher}, F.\,Gantmacher used the
function $-\dfrac{p_1(-u)}{p_0(-u)}$.}
\begin{equation}\label{assoc.function}
\Phi(u)=\displaystyle\frac {p_1(u)}{p_0(u)}.
\end{equation}
\begin{definition}\label{def.associated.function}
We call $\Phi$ the \textit{function  associated with the
polynomial}~$p$.
\end{definition}

Note that the infinite Hurwitz matrix associated with the function $\Phi$, is
the matrix $H_{\infty}(p)$ defined in~\eqref{Hurwitz.matrix.infinite.for.poly},
since $H_{\infty}(p)=H(p_0,p_1)$. We denote the leading principal minors of the matrix $H_{\infty}(p)$
as $\eta_j(p)$, $j=1,2,\ldots.$

The corresponding finite Hurwitz matrix related to the rational function $\Phi$ is the matrix $\mathcal{H}_n(p)$ defined
in~\eqref{HurwitzMatrix}, since $\mathcal{H}_{n}(p)=H_{2l}(p_0,p_1)$ if $n=2l$, and $\mathcal{H}_{n}(p)=H_{2l+1}(p_0,p_1)$ if $n=2l+1$.

\begin{definition}\label{def.Hurwitz.dets}
The leading principal minors of the matrix $\mathcal{H}_n(p)$ are denoted by
$\Delta_j(p)$,
\begin{equation}\label{delta}
\Delta_{j}(p)=
\begin{vmatrix}
a_1&a_3&a_5&a_7&\dots&a_{2j-1}\\
a_0&a_2&a_4&a_6&\dots&a_{2j-2}\\
0  &a_1&a_3&a_5&\dots&a_{2j-3}\\
0  &a_0&a_2&a_4&\dots&a_{2j-4}\\
\vdots&\vdots&\vdots&\vdots&\ddots&\vdots\\
0  &0  &0  &0  &\dots&a_{j}
\end{vmatrix},\quad j=1,\ldots,n,
\end{equation}
with the convention that $a_i=0$ for $i>n$, and
 are called the
\textit{Hurwitz determinants} or the \textit{Hurwitz minors} of
the polynomial~$p$. For simplicity, we set $\Delta_0(p)\equiv1$.
\end{definition}

%%%%%%%%%%%%%%%%%%%%%%%%%%%%%%%%%%%%%%%%%%%%%%%%%%%%%%%%%%%%%%%%%%%%%%%%%%%
\subsection{Stable polynomials}\label{subsection:Hurwitz.stability}
%%%%%%%%%%%%%%%%%%%%%%%%%%%%%%%%%%%%%%%%%%%%%%%%%%%%%%%%%%%%%%%%%%%%%%%%%%%

In this section, we remind the reader some basic and well known facts about stable polynomials.

\begin{definition}\label{def:stable.poly}
A polynomial is called \textit{(Hurwitz) stable} if all its zeros lie in the open left half-plane of the complex plane.
\end{definition}

It is well-known~\cite[Chapter~XV]{Gantmakher} that the polynomial $p$ is stable if and only if the polynomial $p_0(u)$ and $p_1(u)$
have simple, negative, and interlacing zeros, that is between two zeros of one polynomial there lie exactly one zero of the other polynomial.
This fact together with some properties of \textit{R}-functions (see, e.g.,~\cite[Theorem~3.4]{Holtz_Tyaglov}) implies the following result~\cite[Chapter~XV]{Gantmakher} whose proof (due to Barkovsky~\cite{Barkovsky.private}) is given in Appendix
to make the paper self-contained.

\begin{prepos}\label{Propos.Hurw.st.R.func}
The polynomial $p$ defined in~\eqref{main.poly} is
stable if and only if its associated function~$\Phi$
defined in~\eqref{assoc.function} is an \textit{R}-function with
exactly $l$ poles, all of which  are negative, and the limit
$\displaystyle\lim_{u\to\pm\infty}\Phi(u)$ is positive whenever
$n=2l+1$, where the number $l$ is defined in~\eqref{floor.poly.degree}.
\end{prepos}

Proposition~\ref{Propos.Hurw.st.R.func} together with Theorems~\ref{Theorem.R-functions.with.positive.poles.via.infinite.Hurwitz.matrix} and~\ref{Th.Hurwitz.Matrix.Total.Nonnegativity}
imply the next theorem which, besides providing other properties, completely characterizes the total nonnegativity of the Hurwitz matrices of stable polynomials.

\begin{theorem}\label{Theorem.Hurwitz.stable.Hurwitz.matrix.criteria}
Given a polynomial $p$ of degree $n$ as
in~\eqref{main.poly}, the following statements are
equivalent:
\begin{itemize}
\item[1)] The polynomial $p$ is stable;
\item[2)] all Hurwitz minors $\Delta_j(p)$ are positive:
\begin{equation}\label{Hurvitz.det.noneq}
\Delta_1(p)>0,\ \Delta_2(p)>0,\dots,\ \Delta_n(p)>0;
\end{equation}
\item[3)] the determinants $\eta_j(p)$ are positive up to order $n+1$:
\begin{equation}\label{Hurvitz.det.noneq.infinite}
\eta_1(p)>0,\ \eta_2(p)>0,\dots,\ \eta_{n+1}(p)>0;
\end{equation}
\item[4)] the matrix $\mathcal{H}_n(p)$ defined in~\eqref{HurwitzMatrix} is nonsingular and totally
nonnegative;
\item[5)] the matrix $H_{\infty}(p)$ defined in~\eqref{Hurwitz.matrix.infinite.for.poly} is totally nonnegative with the minor $\eta_{n+1}(p)$ being nonzero.
\end{itemize}
\end{theorem}
\noindent Note that the equivalence of $1)$ and $2)$ is the famous
Hurwitz criterion of stability~\cite{Hurwitz} (see also~\cite[Ch.~XV]{Gantmakher}). The implications
$1)\Longrightarrow4)$ and $1)\Longrightarrow5)$ were proved
in~\cite{Asner,{Kemperman}}. The implication $4)\Longrightarrow1)$
was, in fact, proved in~\cite{Asner}. However, the implication
$5)\Longrightarrow1)$ was only mentioned in~\cite{Dyachenko} as a consequence of~\cite[Theroem~3.44]{Holtz_Tyaglov}.

%%%%%%%%%%%%%%%%%%%%%%%%%%%%%%%%%%%%%%%%%%%%%%%%%%%%%%%%%%%%%%%%%%%%%%%%%%%
\subsection{Quasi-stable polynomials}\label{subsection:quasi.stability}
%%%%%%%%%%%%%%%%%%%%%%%%%%%%%%%%%%%%%%%%%%%%%%%%%%%%%%%%%%%%%%%%%%%%%%%%%%%

In this section, we deal with polynomials whose zeros
lie in the~\textit{closed} left half-plane.

\begin{definition}\label{def.quasi-stable.poly}
A polynomial $p$ of degree $n$ defined in~\eqref{main.poly} is called \textit{quasi-stable}
with \textit{the stability index}~$m$, $0\leqslant m\leqslant n$, if
all its zeros lie in the \textit{closed} left half-plane of the
complex plane and the number of zeros of~$p$ on the imaginary
axis, counting multiplicities, equals~$n-m$.
We call the number $n-m$ \textit{the degeneracy index} of the quasi-stable polynomial~$p$.
\end{definition}

Obviously, any stable polynomial is quasi-stable with zero
degeneracy index, that is, it has the smallest degeneracy index
and the largest stability index (which equals the degree of the
polynomial).

\begin{remark}\label{Remark.degeneracy.even}
Note that the degeneracy index $n-m$ is always \textit{even}
due to the condition $p(0)\neq0$ adopted in~\eqref{main.poly}.
\end{remark}

Throughout this section we use the following notation
\begin{equation}\label{floor.poly.stability.index}
r=\left[\dfrac m2\right],
\end{equation}
where $m$ is the stability index of the
polynomial $p$.

Moreover, if $p$ is a quasi-stable polynomial, then
\begin{equation*}\label{Th.stable.poly.and.R-function.1}
p(z)=p_0(z^2)+zp_1(z^2)=g(z^2)q(z)=g(z^2)\left[q_0(z^2)+zq_1(z^2)\right],
\end{equation*}
where $q$ is a stable polynomial, while $g(u)=\gcd(p_0,p_1)$ has only \textit{negative} zeros. Using this
representation of quasi-stable polynomials, one can extend almost
all results of Section~\ref{subsection:Hurwitz.stability} to
quasi-stable polynomials in the same way.

The next theorem is an extended version of
Proposition~\ref{Propos.Hurw.st.R.func}.

\begin{theorem}\label{Th.stable.poly.and.R-function}
The polynomial $p$ defined
in~\eqref{main.poly} is quasi-stable with the stability index~$m$
if and only if its associated function~$\Phi$ defined
in~\eqref{assoc.function} is an \textit{R}-function
of negative type with exactly $r$ poles all of which are
negative, and $\displaystyle\lim_{u\to\pm\infty}\Phi(u)$ is
positive whenever $n$ is odd. The number $r$ is defined
in~\eqref{floor.poly.stability.index}.
\end{theorem}

Now it is also easy to extend
Theorem~\ref{Theorem.Hurwitz.stable.Hurwitz.matrix.criteria} to
quasi-stable polynomials.

\begin{theorem}\label{Th.quasi.stable.Hurwitz.matrix.criteria}
Given a polynomial $p$ of degree $n$ as
in~\eqref{main.poly}, the following statements are
equivalent:
\begin{itemize}
\item[1)] The polynomial $p$ is quasi-stable with the stability index $m$;
\item[2)] the Hurwitz minors $\Delta_j(p)$ are positive up to order
$m$:
\begin{equation}\label{Th.quasi.stable.Hurwitz.matrix.criteria.condition.1}
\Delta_1(p)>0,\ \Delta_2(p)>0,\dots,\ \Delta_{m}(p)>0,\
\Delta_{m+1}(p)=\ldots=\Delta_n(p)=0,
\end{equation}
and $g(u)=\gcd(p_0,p_1)$ has only negative zeros;

\item[3)] the determinants $\eta_j(p)$ are positive up to order
$m+1$:
\begin{equation}\label{Th.quasi.stable.Hurwitz.matrix.criteria.condition.2}
\eta_1(p)>0,\ \eta_2(p)>0,\dots,\ \eta_{m+1}(p)>0,\
\eta_{m+i}(p)=0,\quad i=2,3,\ldots,
\end{equation}
and $g(u)=\gcd(p_0,p_1)$ has only negative zeros;
%\item[4)] the matrix $\mathcal{H}_n(p)$ is totally
%nonnegative and
%
%\begin{equation*}\label{Th.quasi.stable.Hurwitz.matrix.criteria.condition.3}
%\Delta_{n-m}(p)\neq0,\,\Delta_{n-m+1}(p)=\ldots=\Delta_n(p)=0;
%\end{equation*}
%
\item[4)] the matrix $H_{\infty}(p)$ is totally nonnegative and
\begin{equation*}\label{Th.quasi.stable.Hurwitz.matrix.criteria.condition.4}
\eta_{m+1}(p)\neq0,\,\eta_{m+i}(p)=0,\quad i=2,3,\ldots.
\end{equation*}
\end{itemize}
\end{theorem}

\noindent The implication $1)\Longrightarrow4)$ was proved in~\cite{Asner,{Kemperman}}. The implication
$4)\Longrightarrow1)$ is a simple consequence of Theorems~\ref{Th.Hurwitz.Matrix.Total.Nonnegativity} and~\ref{Th.stable.poly.and.R-function}
as it was noticed in~\cite{Dyachenko}. Thus, the equivalence $1)\Longleftrightarrow 4)$ can be rewritten in the following form.
\begin{theorem}\label{Thm.quasi.stable.Hurwitz.matrix.criteria.2}
 A polynomial is quasi-stable if and only if its infinite
Hurwitz matrix is totally nonnegative.
\end{theorem}

Note that in the conditions $2)$ and $3)$ of Theorem~\ref{Th.quasi.stable.Hurwitz.matrix.criteria} we cannot circumvent the condition
that the gcd of the polynomials $p_0(u)$ and $p_1(u)$, the even and odd parts of the polynomial $p$, has only negative zeros.
Indeed, the implications $1)\Longrightarrow 2)$ and $1)\Longrightarrow 3)$ follow from results in~\cite{Asner} and~\cite{Kemperman}, respectively,
since the inequalities~\eqref{Th.quasi.stable.Hurwitz.matrix.criteria.condition.1} and~\eqref{Th.quasi.stable.Hurwitz.matrix.criteria.condition.2}
follow from the total nonnegativity of the Hurwitz matrix $H_{\infty}(p)$ of a quasi-stable polynomial (for properties of totally nonnegative matrices,
see, e.g.,~\cite{Karlin,Ando,Pinkus2010,FallatJohnson}).

However, the inequalities~\eqref{Th.quasi.stable.Hurwitz.matrix.criteria.condition.1} or~\eqref{Th.quasi.stable.Hurwitz.matrix.criteria.condition.2}
without any additional condition imply only that the polynomial
\begin{equation}\label{poly.quotient}
q(z)=\dfrac{p(z)}{g(z^2)},
\end{equation}
where $g(u)=\gcd(p_0,p_1)$, is stable and is of degree $m$. They provide no information about the root
location of the polynomial $g(u)$ at all. For example, the polynomial $p(z)=(z+1)(z^4+1)$ satisfies the
inequalities~\eqref{Th.quasi.stable.Hurwitz.matrix.criteria.condition.1} with $m=1$. Its finite Hurwitz matrix has the form
\begin{equation*}
H_5(p)=
\begin{pmatrix}
1&0&1&0&0\\
1&0&1&0&0\\
0&1&0&1&0\\
0&1&0&1&0\\
0&0&1&0&1\\
\end{pmatrix}.
\end{equation*}
But this matrix is not totally nonnegative.

Taking into account Theorem~\ref{Thm.quasi.stable.Hurwitz.matrix.criteria.2}, one can suppose that the assumption of the total nonnegativity of the finite Hurwitz matrix of a polynomial can imply the stability of the polynomial. As we announce in the introduction, this supposition is true, and the total nonnegativity of the finite Hurwitz matrix put some
restrictions on the roots of the polynomial. But it also does not imply quasi-stability of the polynomial. As we mentioned in the introduction that
Asner~\cite{Asner} provided as a counterexample the polynomial
 $p(z)=z^4+198z^2+10201$ with zeros $\pm1\pm i10$ whose finite Hurwitz
matrix
\begin{equation*}
H_4(p)=
\begin{pmatrix}
0&0&0&0\\
1&198&10201&0\\
0&0&0&0\\
0&1&198&10201\\
\end{pmatrix}
\end{equation*}
is totally nonnegative. Note that $p$ satisfies the conditions of Theorem~\ref{Theorem.suf.cond}.

The main reason of this phenomenon is the following. As we show below (see~\eqref{finite.Hurw.matr.gcd.factrozation}), the finite Hurwitz
matrix of the given polynomial $p$ can be factorized as follows $\mathcal{H}_n(p)=H_n(q)\mathcal{T}_n(g)$, where
$H_n(q)$ is a truncation of the infinite Hurwitz matrix of the polynomial $q$ defined in~\eqref{poly.quotient}, while $\mathcal{T}_n(g)$ is a finite triangular Toeplitz matrix consisting of the coefficients of the polynomial $g(u)=\gcd(p_0,p_1)$; both matrices are to be defined in Section~\ref{section:proofs}. As the previous example shows,
it is possible to find a polynomial $p$ such that the matrix $\mathcal{T}_n(g)$ is totally nonnegative and $q$ is stable.
In this case, $\mathcal{H}_n(p)$ is totally nonnegative as a product of matrices of such a type while the infinite Hurwitz
matrix $H_{\infty}(p)$ is not totally non-negative if $g(u)$ has positive or/and non-real zeros. However, if $\deg g=1$, $g(u)$
cannot have nonnegative roots in the case when $\mathcal{H}_n(p)$ is totally nonnegative. So we can generalize Theorem~\ref{Theorem.Hurwitz.stable.Hurwitz.matrix.criteria} as follows.

\begin{theorem}\label{Th.quasi.stable.Hurwitz.matrix.criteria.order.2}
Given a polynomial $p$ of degree $n$ as
in~\eqref{main.poly}, the following statements are
equivalent:
\begin{itemize}
\item[1)] The polynomial $p$ is quasi-stable with the stability index at least $n-2$;
\item[2)] the matrix $\mathcal{H}_n(p)$ is totally
nonnegative and $\Delta_{n-2}(p)\neq0$.
\end{itemize}
\end{theorem}

In the next section we find the location of the roots of polynomials whose finite Hurwitz matrices are totally nonnegative.
Throughout the next sections by
$$
A\begin{pmatrix}
i_1&i_2&\cdots&i_k\\
j_1&j_2&\cdots&j_k
\end{pmatrix}
$$
we denote the minor of a matrix $A$ (finite or infinite) formed with its rows $i_1$, $i_2$, \ldots, $i_k$  and columns $j_1$, $j_2$, \ldots, $j_k$.
We suppose here that $1\leqslant i_1<i_2<\cdots<i_k$, and $1\leqslant j_1<j_2<\cdots<j_k$.

\setcounter{equation}{0}

%%%%%%%%%%%%%%%%%%%%%%%%%%%%%%%%%%%%%%%%%%%%%%%%%%%%%%%%%%%%%%%%%%%%%%%%%%%%%%%%%%%%%%%%%%%%%%%%%%%%%
\section{Proofs of Theorems~\ref{Theorem.nes.cond} and~\ref{Theorem.suf.cond}}\label{section:proofs}
%%%%%%%%%%%%%%%%%%%%%%%%%%%%%%%%%%%%%%%%%%%%%%%%%%%%%%%%%%%%%%%%%%%%%%%%%%%%%%%%%%%%%%%%%%%%%%%%%%%%%

In this section, we find necessary and sufficient conditions for the total nonnegativity
of the finite Hurwitz matrix of a given polynomial. Throughout the section we suppose
additionally that the size of the last nonzero leading principal minor of the finite
Hurwitz matrix is known in advance, and establish all results under this additional
condition. Then Theorems~\ref{Theorem.nes.cond} and~\ref{Theorem.suf.cond} follow
from the results of this section by putting the size of the last nonzero leading
principal minor to be maximal or minimal.

The first of our results deals with rank of finite Hurwitz matrices. Obviously, given a polynomial $p$,
rank of the matrix $\mathcal{H}_n(p)$ equals $n$ whenever $\Delta_n(p)\neq0$. Let us generalize this fact.

\begin{lemma}\label{Lemma.finite.Hurw.matr.fact}
Let $p$ be the polynomial defined in~\eqref{main.poly}. If $\Delta_m(p)\neq0$ and
$\Delta_j(p)=0$, $j=m+1,\ldots,n$, for some number $m$, $0\leqslant m\leqslant n$, then
rank of the finite Hurwitz matrix $\mathcal{H}_n(p)$ is $\dfrac{n+m}2$.
\end{lemma}
\begin{proof}
The case $m=n$ was just mentioned above. If $m=0$, then the claim of the lemma is trivial.

Now the assertion of the theorem in the case $1\leqslant m\leqslant n-1$ follows from Theorems~\ref{Thm.fin.Hurw.matr.factor.1} and~\ref{Thm.fin.Hurw.matr.factor.2}.
\end{proof}
\begin{remark}
Note that the number $\dfrac{n+m}2=m+\dfrac{n-m}{2}$ is always integer, since $n-m$ is an even number according to Remark~\ref{Remark.degeneracy.even}.
\end{remark}

For our next result, we use the following auxiliary definition and facts.
\begin{definition}
Given a polynomial $g$, if all minors of order $\leqslant r$ of the infinite Toeplitz matrix
$\mathcal{T}(g)$ defined in~\eqref{Toeplitz.infinite.matrix} are nonnegative, then the sequence of the
coefficients of the polynomial $g$ and the matrix $\mathcal{T}(g)$ are called $r$-\textit{times
nonnegative} or $r$-\textit{nonegative}. If $\mathcal{T}(g)$ is totally
nonnegative, then the sequence of the coefficients of the
polynomial $g$ is called \textit{totally
nonnegative}~\cite{Schoenberg.m.pos.1,Schoenberg.m.pos.2}\footnote{In~\cite{Schoenberg.m.pos.1,Schoenberg.m.pos.2},
such sequences are called \textit{r}-\textit{positive} and \textit{totally positive}, respectively.}. The
functions generating \textit{r}-nonnegative (totally nonnegative) sequences are
usually denoted by $PF_r$ ($PF_\infty$). %The notation $TN_r$ is used for $r$-nonnegative matrices.
\end{definition}

For polynomials in the class $PF_r$, I.\,Schoenberg established the following theorem.
\begin{theorem}[Schoenberg~\cite{Schoenberg.m.pos.2}]\label{Thm.Schoenberg.3}
The polynomial $g(u)=g_0u^l+\cdots+g_l$, $g_0\neq0$, belongs to the class $PF_r$ if and only if the $r\times(r+l)$ matrix
\begin{equation*}\label{matrix.T.truncated}
T_r\stackrel{def}{=}\begin{pmatrix}
g_0&\cdots&g_l&      & 0 \\
   &\ddots&   &\ddots&   \\
 0 &      &g_0&\cdots&g_l
\end{pmatrix}
\end{equation*}
is totally nonnegative.
\end{theorem}

Moreover, it is easy to see that the total nonnegativity of the matrix $T_r$ is equivalent to the nonnegativity of the minors of order $r$ of $T_r$, since
every minor of the matrix $T_r$ of order less than $r$ formed with consecutive rows is either zero or a product of a minor of order $r$ and a positive constant of
the form $g_0^{-k}$, $k\in\mathbb{N}$. Indeed, for any $s$, $1\leqslant s\leqslant r$, and $1\leqslant j_1<j_2<\cdots<j_s\leqslant r$ one has
$$
T_r
\begin{pmatrix}
i&i+1&\cdots&i+s-1\\
j_1&j_2&\cdots&j_s
\end{pmatrix}=0,
$$
for any $i>j_1$. If $i\leqslant j_1$, the following identity holds
\begin{equation}\label{minor.1}
T_r
\begin{pmatrix}
i&i+1&\cdots&i+s-1\\
j_1&j_2&\cdots&j_s
\end{pmatrix}=
T_r
\begin{pmatrix}
i+t&i+t+1&\cdots&r\\
j_1+t&j_2+t&\cdots&j_s+t
\end{pmatrix},
\end{equation}
where $t=r+1-i-s$. Now since
$$
T_r
\begin{pmatrix}
1&2&\cdots&k\\
1&2&\cdots&k
\end{pmatrix}=g_0^{k},\qquad k=1,\ldots,r,
$$
we finally can conclude from~\eqref{minor.1} that
\begin{equation}\label{minor.2}
\small{
T_r
\begin{pmatrix}
i&i+1&\cdots&i+s-1\\
j_1&j_2&\cdots&j_s
\end{pmatrix}=\dfrac{1}{g_0^{r+s}}\cdot
T_r
\begin{pmatrix}
1&2&\cdots&i+t-1&i+t&i+t+1&\cdots&r\\
1&2&\cdots&i+t-1&j_1+t&j_2+t&\cdots&j_s+t
\end{pmatrix}.}
\end{equation}

This formula together with Theorem~\ref{Thm.Schoenberg.3} and~\cite[Theorem~2.1]{Ando} imply the following fact.
\begin{lemma}\label{Lemma.PF_r}
The polynomial $g(u)=g_0u^l+\cdots+g_l$, $g_0\neq0$, belongs to the class $PF_r$ if and only if all the minors of the matrix $T_r$ of order $r$ are nonnegative.
\end{lemma}

Now we are in a position to establish a fact which is a basic tool in the proofs of our main results\footnote{In fact,
Lemma~\ref{Lemma.finite.Hurw.matr.tot.neg} proves Conjecture~3.49 in~\cite{Holtz_Tyaglov}.}.

\begin{lemma}\label{Lemma.finite.Hurw.matr.tot.neg}
Let $p$ be the polynomial defined in~\eqref{main.poly}. Its finite Hurwitz matrix $\mathcal{H}_n(p)$
is totally nonnegative with
\begin{equation}\label{Lemma.finite.Hurw.matr.tot.neg.cond}
\Delta_m(p)\neq0\quad\text{and}\quad\Delta_{m+1}(p)=0,
\end{equation}
if and only if $p(z)=q(z)g(z^2)$ with $\deg q=m$,
where $q$ is a stable polynomial and $g\in PF_{\tfrac{n+m}2}$.
\end{lemma}
\begin{proof}
Let $p(z)=q(z)g(z^2)$, where
$$
q(z)=b_0z^m+b_1z^{m-1}+\cdots+b_m,\qquad b_0>0,
$$
is a stable polynomial and $g\in PF_{\tfrac{n+m}2}$.
The inequalities~\eqref{Lemma.finite.Hurw.matr.tot.neg.cond} follow from~\eqref{degeneracy.R.func}. Furthermore, by Theorems~\ref{Thm.fin.Hurw.matr.factor.1} and~\ref{Thm.fin.Hurw.matr.factor.2} the matrix $\mathcal{H}_n(p)$ can be factorized as follows
\begin{equation}\label{finite.Hurw.matr.gcd.factrozation}
\mathcal{H}_n(p)=H_n(q)\mathcal{T}_n(g),
\end{equation}
where $H_n(q)$ is the $n\times n$ principal submatrix of the infinite Hurwitz matrix
$H_{\infty}(q)$ indexed by rows (and columns) $2$ through $n+1$, and the matrix
$\mathcal{T}_{n}(g)$ is the $n\times n$ leading principal submatrix of the matrix~$\mathcal{T}(g)$.
By Theorem~\ref{Thm.quasi.stable.Hurwitz.matrix.criteria.2}, the matrix $H_{\infty}(q)$ is totally nonnegative,
so is its submatrix $H_{n}(q)$. Moreover, all minors of $\mathcal{T}(g)$ of order $\leqslant\frac{n+m}2$ are nonnegative by assumption.
Therefore, by the Binet-Cauchy formula, all the minors of $\mathcal{H}_n(p)$ of order $\leqslant\frac{n+m}2$ are nonnegative.
Now since $\mathrm{rank}\,\mathcal{H}_n(p)=\frac{n+m}2$ according to Lemma~\ref{Lemma.finite.Hurw.matr.fact}, we obtain that $\mathcal{H}_n(p)$ is totally nonnegative.

\vspace{2mm}

Conversely, suppose that $\mathcal{H}_n(p)$ is totally nonnegative and condition~\eqref{Lemma.finite.Hurw.matr.tot.neg.cond} holds.
By~\cite[Lemma~5, Ch.~XIII]{Gantmakher} (see also~\cite[Corollary 9.1, Ch.~2]{Karlin}), one obtains
\begin{equation}\label{Lemma.finite.Hurw.matr.tot.neg.proof.1}
\begin{array}{l}
\Delta_j(p)>0,\qquad j=1,\ldots,m,\\
\\
\Delta_j(p)=0,\qquad j=m+1,\ldots,n.
\end{array}
\end{equation}
As we mentioned in Section~\ref{section:R.function}, %(see~\eqref{degeneracy.R.func}),
this means that $p$ can be factorized as follows
$$
p(z)=q(z)g(z^2)
$$
with $\deg q = m$ and $\deg g=\frac{n-m}2$. Now from the factorization~\eqref{finite.Hurw.matr.gcd.factrozation} and from the Binet-Cauchy formula we obtain that
\begin{equation}\label{lead.princ.min}
\mathcal{H}_n(p)
\begin{pmatrix}
1&2&\cdots&k\\
1&2&\cdots&k
\end{pmatrix}=g_0^{k}\cdot H_n(q)
\begin{pmatrix}
1&2&\cdots&k\\
1&2&\cdots&k
\end{pmatrix},\quad k=1,\ldots,n,
\end{equation}
due to the special structure of $\mathcal{T}_n(g)$. Consequently, we have that $\Delta_i(q) >0$, $i=1, \ldots, m$,
since we can always choose $g_0$, the leading coefficient of $g(u)$, to be positive. So the polynomial $q$ is stable, and the matrix $H_{\infty}(q)$
is totally nonnegative according to Theorem~\ref{Theorem.Hurwitz.stable.Hurwitz.matrix.criteria}.
By Lemma~\ref{Lemma.PF_r}, it suffices to show now that
\begin{equation*}
\mathcal{T}_n(g)
\begin{pmatrix}
1&2&\cdots&\frac{n+m}2\\
j_1&j_2&\cdots&j_{\tfrac{n+m}2}
\end{pmatrix}\geqslant0,
\end{equation*}
for any $j_i$ such that $1\leqslant j_1<j_2<\cdots<j_{\tfrac{n+m}2}\leqslant n$.

Since the only nonzero minor of the matrix $H_n(q)$ formed with rows $1$, $2$, \ldots, $m$, $m+2$, $m+4$, \ldots, $n$, is the one which is
formed with columns $1$, $2$, \ldots, $\frac{n+m}2$, the Binet-Cauchy formula implies
\begin{equation}\label{eqprank}
\begin{array}{l}
\mathcal{T}_n(g)
\begin{pmatrix}
1&2&\cdots&\frac{n+m}2\\
j_1&j_2&\cdots&j_{\tfrac{n+m}2}
\end{pmatrix}=
\dfrac{\mathcal{H}_n(p)
\begin{pmatrix}
1&2&\cdots&m&m+2&m+4&\cdots&n\\
j_1&j_2&\cdots&j_m&j_{m+1}&j_{m+2}&\cdots&j_{\tfrac{n+m}2}
\end{pmatrix}}{H_n(q)
\begin{pmatrix}
1&2&\cdots&m&m+2&m+4&\cdots&n\\
1&2&\cdots&m&m+1&m+2&\cdots&\tfrac{n+m}2
\end{pmatrix}}\\
\\
=\dfrac{\mathcal{H}_n(p)
\begin{pmatrix}
1&2&\cdots&m&m+2&m+4&\cdots&n\\
j_1&j_2&\cdots&j_m&j_{m+1}&j_{m+2}&\cdots&j_{\tfrac{n+m}2}
\end{pmatrix}}{b_m^{\tfrac{n-m}2}\Delta_m(q)}\geqslant0,
\end{array}
\end{equation}
where $\Delta_m(q)>0$ due to the stability of the polynomial $q$. Thus, $g\in PF_{\tfrac{n+m}2}$, as desired.

\end{proof}

To obtain our final results of this section which imply Theorems~\ref{Theorem.nes.cond} and~\ref{Theorem.suf.cond}, we remind the
reader the following remarkable theorems due to Schoenberg~\cite{Schoenberg.m.pos.1,Schoenberg.m.pos.2}.

\begin{theorem}[Schoenberg~\cite{Schoenberg.m.pos.2}]\label{Thm.Schoenberg.1}
Given a real polynomial $g(u)$ of degree $r$, if the matrix $\mathcal{T}(g)$ is
$k$-times nonnegative, then $g(u)$ has no zeros in the sector

\begin{equation}\label{Schoenberg.sector.1}
|\arg u|<\dfrac{\pi k}{r+k-1}\,.
\end{equation}
The constant in~\eqref{Schoenberg.sector.1} is sharp.
\end{theorem}
Indeed, in~\cite{Schoenberg.m.pos.2} it was shown that the polynomial
\begin{equation}\label{Schoenberg.sharp.1}
g(z)=\prod_{j=1}^{r}\left(z+e^{i\theta(r-2j+1)}\right),\qquad \theta=\dfrac{\pi}{r+k-1},
\end{equation}
belongs to the class $PF_k$, but it has zeros on the border of the sector~\eqref{Schoenberg.sector.1}.
These zeros are $-e^{\pm i\theta(r-1)}$.

\begin{theorem}[Schoenberg~~\cite{Schoenberg.m.pos.1}]\label{Thm.Schoenberg.2}
Given a real polynomial $g(u)$, if all zeros of $g(u)$ lie in the sector
\begin{equation}\label{Schoenberg.sector.2}
\pi-\dfrac{\pi}{k+1}\leqslant\arg u\leqslant\pi+\dfrac{\pi}{k+1}\,,
\end{equation}
then $g\in PF_k$.
The constant in~\eqref{Schoenberg.sector.2} is sharp.
\end{theorem}

In~\cite{Schoenberg.m.pos.1,Schoenberg.m.pos.2} it was shown that a polynomial $g(u)$ of the form
\begin{equation}\label{counter.example.poly.2}
(u+ce^{i\theta})(u+ce^{-i\theta}),\quad c>0,
\end{equation}
belongs to $PF_k$ if and only if
\begin{equation}\label{sector.theta}
0\leqslant\theta\leqslant\dfrac{\pi}{k+1}.
\end{equation}
Thus, any product of polynomials of the form~\eqref{counter.example.poly.2} with $\theta$ exterior to the interval~\eqref{sector.theta}
does not belong to $PF_k$.

Now we are in a position to prove the main results of this section.

\begin{theorem}\label{Theorem.nes.cond.m}
Let $p$ be the polynomial of degree $n$ given in~\eqref{main.poly}. If its finite Hurwitz matrix $\mathcal{H}_n(p)$
is totally nonnegative and $\Delta_m(p)\neq0$, $\Delta_{m+1}(p)=0$ for some $m$, $0\leqslant m\leqslant n-2$, then $p$ has no zeros in the sector
\begin{equation}\label{sector.1}
|\arg z|<\dfrac{\pi}{4}\cdot\dfrac{n+m}{n-1}.
\end{equation}
The constant in~\eqref{sector.1} is sharp.
\end{theorem}
\begin{proof}
Let the finite Hurwitz matrix $\mathcal{H}_n(p)$ of the polynomial $p$
be totally nonnegative and $\Delta_m(p)\neq0$, $\Delta_{m+1}(p)=0$ for some $m$, $0\leqslant m\leqslant n-2$.
Then by Lemma~\ref{Lemma.finite.Hurw.matr.tot.neg}, one has $p(z)=q(z)g(z^2)$, where $q$ is stable
and $g\in PF_{\tfrac{n+m}2}$. Consequently, by Theorem~\ref{Thm.Schoenberg.1} all zeros of the polynomial $g(u)$
lie outside the sector
\begin{equation}\label{sector.for.section.6}
|\arg u|<\dfrac{\pi}2\cdot\dfrac{n+m}{n-1},
\end{equation}
so all the zeros of $g(z^2)$ lie outside the sector~\eqref{sector.1}. Moreover, for any $m$, $0\leqslant m\leqslant n-2$,
the finite Hurwitz matrix of the polynomial
$$
p(z)=(z+1)^m\prod_{j=1}^{\tfrac{n-m}{2}}\left(z^2+e^{i\tfrac{\pi}{2}\tfrac{n-m-4j+2}{n-1}}\right)
$$
whose zeros lie outside the sector~\eqref{sector.1} and on the border of this sector, is totally nonnegative. This follows from Lemma~\ref{Lemma.finite.Hurw.matr.tot.neg} and from Schoenberg's example~\eqref{Schoenberg.sharp.1}. Thus, the angle on the right-hand side of~\eqref{sector.1} cannot be improved.
\end{proof}
Note that in the case $m=n-2$ the polynomial $p$ is quasi-stable that corresponds to Theorem~\ref{Th.quasi.stable.Hurwitz.matrix.criteria.order.2}. So in
the following theorem we suppose that $m\leqslant n-4$.

\begin{theorem}\label{Theorem.suf.cond.m}
Let the polynomial of degree $n\geqslant 4$ given in~\eqref{main.poly} have no zeros in the sector
\begin{equation}\label{sector.3}
|\arg z|<\dfrac{\pi}{2}\cdot\dfrac{n+m}{n+m+2},
\end{equation}
for some number $m$, $0\leqslant m\leqslant n-4$, and satisfy the ''reflection property'': $p(-\lambda)=0$ for any $\lambda$ such that $p(\lambda)=0$ and $\Re\lambda>0$,
then the finite Hurwitz matrix $\mathcal{H}_n(p)$ of the polynomial $p$ is totally nonnegative
with $\Delta_m(p)\neq0$ and $\Delta_{m+1}(p)=0$. The constant in~\eqref{sector.3} is sharp.

The polynomial $p$ of degree $n$, $1\leqslant n\leqslant3$, is quasi-stable if and only if $\mathcal{H}_n(p)$ is totally nonnegative.
\end{theorem}
\begin{proof}
From the condition of the theorem, it follows that $p$ can be factorized as $p(z)=q(z)g(z^2)$, where
$q$ is stable, and $g(z^2)$ has no zeros in the sector~\eqref{sector.3}. This means that $g(u)$
has all zeros in the sector
\begin{equation}\label{sector.5}
\pi-\dfrac{2\pi}{n+m+2}\leqslant\arg u\leqslant\pi+\dfrac{2\pi}{n+m+2},
\end{equation}
so by Theorem~\ref{Thm.Schoenberg.2}, $g\in PF_{\tfrac{n+m}2}$. Now from Lemma~\ref{Lemma.finite.Hurw.matr.tot.neg}
it follows that $\mathcal{H}_n(p)$ is totally nonnegative.

Furthermore, if there exists a number $\lambda$, $\Re\lambda>0$, such that $p(\lambda)=0$ and $p(-\lambda)\neq0$, then\footnote{$g(u)$ can be a constant.}
$p(z)=q(z)g(z^2)$ and $q(\lambda)=0$. So, $q$ is not stable, and there exists $\Delta_k(q)<0$ or $\Delta_k(q)=0$ and $\Delta_{k+1}(q)\neq0$ for
some $k$, $1\leqslant k\leqslant m-1$. Since $\Delta_k(p)=g_0^k\Delta_k(q)$ by~\eqref{lead.princ.min},
we obtain that $\mathcal{H}_n(p)$ is not totally nonnegative, a contradiction.

Also, if we suppose that the polynomial $p$ satisfies the reflection property,
and that it has zeros in the sector~\eqref{sector.3} for some number $m$, $1\leqslant m\leqslant n-2$, then $p(z)=q(z)g(z^2)$, where $q$ is stable, but $g(u)$ has
roots outside the sector~\eqref{sector.5}. Thus, according to Theorem~\ref{Thm.Schoenberg.2}, $g\not\in PF_{\tfrac{n+m}2}$,
so $\mathcal{H}_n(p)$ is not totally nonnegative by Lemma~\ref{Lemma.finite.Hurw.matr.tot.neg}. Therefore, the angle in~\eqref{sector.3} is sharp.

Finally, for $1\leqslant n \leqslant3$, the statement of the theorem follows immediately from Lemma~\ref{Lemma.finite.Hurw.matr.tot.neg} and Theorem~\ref{Th.quasi.stable.Hurwitz.matrix.criteria.order.2}.
\end{proof}

For example, the polynomial
$$
p(z)=(z+1)^m(z^2+1)^{r_1}(z^4+2z^2\cos\theta+1)^{r_2},
$$
where $r_1=\tfrac{n-m}2-2\left\lfloor\tfrac{n-m}4\right\rfloor$, $r_2=\left\lfloor\tfrac{n-m}4\right\rfloor$,  and
$$
\theta=\dfrac{\pi}{n+m+2}+\varepsilon,
$$
with $\varepsilon>0$ arbitrarily small, has zeros inside the sector~\eqref{sector.5} but arbitrary close to its border (due to~$\varepsilon$).
According to results of Schoenberg~\cite{Schoenberg.m.pos.1} and Lemma~\ref{Lemma.finite.Hurw.matr.tot.neg}, the finite Hurwitz matrix of $p$ is not
totally nonnegative. Thus, the angle on the right-hand side of~\eqref{sector.5} cannot be improved.

If we take $m=0$ for even $n$, and $m=1$ for odd $n$ in Theorem~\ref{Theorem.nes.cond.m}, we get Theorem~\ref{Theorem.nes.cond}. Also, putting
$m=n-4$ in Theorem~\ref{Theorem.suf.cond.m}, we get Theorem~\ref{Theorem.suf.cond}.

\setcounter{equation}{0}

\section{Spectral properties of totally nonnegative finite Hurwitz matrices}\label{section:eigenspace}

In this section, we extend the results of Asner~\cite{Asner} and Lehnigk~\cite{Lehnigk} on the spectrum of totally nonnegative finite Hurwitz matrices and on
their eigenspaces.

\begin{theorem}\label{Thm.eigenspaces}
Let $p$ be the polynomial defined in~\eqref{main.poly} and let its finite Hurwitz matrix $\mathcal{H}_n(p)$ be totally nonnegative
Then $\mathcal{H}_n(p)$ has $\frac{n+m}{2}$ positive eigenvalues and $\frac{n-m}{2}$ zero eigenvalues, where the number $m$, $0\leqslant m\leqslant n$,
is such that
\begin{equation}\label{Thm.eigenspaces.cond}
\Delta_m(p)\neq0\quad\text{and}\quad\Delta_{m+1}(p)=0.
\end{equation}
The eigenspaces of the positive eigenvalues are one-dimensional, and the dimension of the eigenspace of the zero eigenvalue is $\frac{n-m}2$.
All positive eigenvalues of $\mathcal{H}_n(p)$ are simple with the possible exception of exactly one which is $p(0)$ and has multiplicity 2.
%but possibly one that can be of multiplicity $2$. This double eigenvalue is equal to $p(0)$ (if any).
\end{theorem}
\begin{proof}
The case $m=n$ was considered in~\cite{Asner} and~\cite{Lehnigk}, so in what follows we consider $m\leqslant n-2$.

If $\mathcal{H}_n(p)$ is totally nonnegative, then from~\cite[Lemma~5, Ch.~XIII]{Gantmakher} it follows that there exists a number~$m$, $0\leqslant  m\leqslant n$, such that
the condition~\eqref{Thm.eigenspaces.cond} holds. Now by Lemma~\ref{Lemma.finite.Hurw.matr.fact}, rank of the matrix~$\mathcal{H}_n(p)$ equals $\frac{n+m}{2}$. Thus,
the polynomial $p$ can be factorized as $p(z)=q(z)g(z^2)$, where $q(z)=b_0z^{m}+b_1z^{m-1}+\cdots+b_m$, $b_0>0$, is stable, and $g(u)\in PF_{\tfrac{n+m}{2}}$.
By Theorem~\ref{Thm.Schoenberg.1}, the polynomial~$g(u)$ has no zeros in the sector~\eqref{sector.for.section.6},
%
%$$
%|\arg u|<\dfrac{\pi}2\cdot\dfrac{n+m}{n-1},
%$$
%
in particular, $g(u)$ is stable, and $\Delta_j(g)>0$, $j=1,\ldots,\frac{n-m}2$. So from~\eqref{finite.Hurw.matr.gcd.factrozation} and~\eqref{eqprank} one has
\begin{equation}\label{Thm.eigenspaces.proof.1}
\begin{array}{l}
\mathcal{H}_n(p)
\begin{pmatrix}
1&2&\cdots&m&m+2&m+4&\cdots&n\\
1&2&\cdots&m&m+2&m+4&\cdots&n\\
\end{pmatrix}=\\
\\
b_{m}^{\frac{n-m}{2}}\Delta_{m}(q)\mathcal{T}_n(g)
\begin{pmatrix}
1&2&\cdots&m&m+1&m+2&\cdots&\frac{n+m}2\\
1&2&\cdots&m&m+2&m+4&\cdots&n\\
\end{pmatrix}=\\
\\
b_{m}^{\frac{n-m}{2}}\cdot\Delta_{m}(q)\cdot g_0^m\cdot \mathcal{T}_n(g)
\begin{pmatrix}
m+1&m+2&\cdots&\frac{n+m}2\\
m+2&m+4&\cdots&n\\
\end{pmatrix}=\\
\\
b_{m}^{\frac{n-m}{2}}\cdot g_0^m\cdot\Delta_{m}(q)\mathcal{T}_n(g)
\begin{pmatrix}
1&2&\cdots&\frac{n-m}2\\
2&4&\cdots&n-m\\
\end{pmatrix}=b_{m}^{\frac{n-m}{2}}\cdot g_0^m\cdot\Delta_{m}(q)\cdot\Delta_{\tfrac{n-m}2}(g)>0.
\end{array}.
\end{equation}

Therefore, the size of the largest nonsingular principal submatrix of $\mathcal{H}_n(p)$ equals $\frac{n+m}{2}$, since
$\mathrm{rank}\,\mathcal{H}_n(p)=\frac{n+m}{2}$. Thus,  we obtain that the algebraic multiplicity of the zero eigenvalue
coincides with its geometric multiplicity and equals $\frac{n-m}{2}=n-\frac{n+m}{2}$ (see, e.g.,~\cite[p.107]{FallatJohnson}). Consequently,
the totally nonnegative matrix $\mathcal{H}_n(p)$ has exactly $\frac{n+m}{2}$ positive eigenvalues.

In the same way as in~\cite{Asner} and~\cite{Lehnigk}, one can prove that only one positive eigenvalue of $\mathcal{H}_n(p)$ can be multiple.
It must be equal $p(0)$ (if any) with multiplicity $2$. Indeed, the matrix $\mathcal{H}_n^{(n-1)}(p)$ obtained from $\mathcal{H}_n(p)$ by deleting the last row and column is an irreducible totally nonnegative matrix and so its positive eigenvalues are simple and distinct by~\cite[Theorem~5.4.5]{FallatJohnson}.
Moreover, the spectrum of $\mathcal{H}_n(p)$ equals the spectrum of $\mathcal{H}_n^{(n-1)}(p)$ together with $a_n=p(0)$. Therefore, only $a_n$ can be a multiple eigenvalue of $\mathcal{H}_n(p)$ (if any). Its  algebraic multiplicity equals two. Let us prove that its geometric multiplicity is one. To do this, we need to prove that
rank of the matrix
\begin{equation}\label{matrix.B}
B_n=\mathcal{H}_n(p)-a_nI,
\end{equation}
is $n-1$, where $I$ denotes the identity matrix. Consider the minor
\begin{equation}\label{minor.3}
B_n\begin{pmatrix}
2&3&\cdots&n\\
1&2&\cdots&n-1\\
\end{pmatrix}=a_0
B_n\begin{pmatrix}
3&4&\cdots&n\\
2&3&\cdots&n-1\\
\end{pmatrix}.
\end{equation}
In the same way as in~\cite{Lehnigk}, that is, using Laplace expansion, one obtains that the minor~\eqref{minor.3}
can be expressed as the sum of $2^{n-3}$ determinants. One of these determinants is $\Delta_{n-2}(p)$.
Let us look at the remaining $2^{n-3}-1$ determinants. Each one of them can be expressed
as a product of an integral power of~$a_n$ and a minor of the matrix $\mathcal{H}_n(p)$. Since
$\mathcal{H}_n(p)$ is totally nonnegative, all its minors are nonnegative, and all the coefficients of $p$ are
positive~\cite{Radke}. From the structure of the matrix~\eqref{matrix.B} it follows that in each of the
aforementioned $2^{n-3}-1$ determinants, $a_n$ stands at a place such
that its cofactor has the sign factor~$-1$. Let us prove that one of these determinants is positive, that is,
that the minor on the right-hand side of~\eqref{minor.3} has the form
\begin{equation}\label{minor.3.expansion}
B_n\begin{pmatrix}
3&4&\cdots&n\\
2&3&\cdots&n-1\\
\end{pmatrix}=\Delta_{n-2}(p)+(2^{n-3}-2)\text{ nonnegative terms}+\delta_n,
\end{equation}
where the minor $\delta_n$ is positive. We prove it by induction, and to do this, we need to
write explicitly the minors $\delta_n$ for $n=4,\ldots,8$. They have the form
$$
\delta_4=a_4a_0,\qquad\delta_5=a_5a_1^2,\qquad\delta_6=a_6^2a_0a_1,\qquad\delta_7=a_7^3a_0a_1,\qquad\delta_8=a_8^4a_0^2.
$$
It is easy to check that such minors exist in the Laplace expansion of the minor~\eqref{minor.3.expansion}
for the corresponding~$n$. Now we prove that in general, the following recursive formul\ae\ hold for $l\geqslant 5$
\begin{equation}\label{recursion}
\begin{array}{l}
\delta_{2l-1}=a_{2l-1}^{l-2}a_1\delta_{l},\\
\\
\delta_{2l}=a_{2l}^{l-1}a_0\delta_{l}.
\end{array}
\end{equation}
The first formula is true for $n=9,10$ as it can easily be seen. Suppose that it is true for all $n(\geqslant9)$ up to some number $N-1$, and consider $n=N$.
Let $N=2k-1\geqslant11$. From the structure of the minor~\eqref{minor.3.expansion} it can be seen that
in its Laplace expansion there exists a minor of the form
$$
a_N^{k-2}a_1B_N\begin{pmatrix}
3&4&\cdots&k\\
2&3&\cdots&k-1\\
\end{pmatrix}.
$$
But the considered minor of the matrix $B_N$ has the same structure as the corresponding minor of the matrix $B_k$
with $a_N$ instead of $a_k$, so we obtain
$$
B_N\begin{pmatrix}
3&4&\cdots&N\\
2&3&\cdots&N-1\\
\end{pmatrix}=a_N^{k-2}a_1B_N\begin{pmatrix}
3&4&\cdots&k\\
2&3&\cdots&k-1\\
\end{pmatrix}=a_{N}^{k-2}a_1\delta_{k}+\text{nonnegative terms}.
$$
Thus, the first formula~\eqref{recursion} is proved. The second one can be proved in a similar
way for $n=2l\geqslant12$. Consequently, the matrix~\eqref{matrix.B} has rank $n-1$, as required.
\end{proof}

Let us illustrate the theorem by an example of a polynomial whose finite Hurwitz matrix has a double non-zero eigenvalue.
The polynomial
$$
p(z)=z^5+(2+\sqrt{2})z^4+\sqrt{2}z^3+(2+2\sqrt{2})z^2+z+2+\sqrt{2},
$$
has the zeros $-2-\sqrt{2}$, $\pm\exp\left(\pm i\dfrac{3\pi}{8}\right)$, so $m=1$. The finite Hurwitz matrix of the polynomial $p$ has the form
$$
\mathcal{H}_5(p)=
\begin{pmatrix}
2+\sqrt{2}&2+2\sqrt{2}&2+\sqrt{2}&0&0\\
1&\sqrt{2}&1&0&0\\
0&2+\sqrt{2}&2+2\sqrt{2}&2+\sqrt{2}&0\\
0&1&\sqrt{2}&1&0\\
0&0&2+\sqrt{2}&2+2\sqrt{2}&2+\sqrt{2}\\
\end{pmatrix}.
$$
By Lemma~\ref{Lemma.finite.Hurw.matr.fact}, rank of this matrix is equal to $3$, and according to Theorem~\ref{Theorem.suf.cond.m},
it is totally nonnegative. The Jordan form of the matrix $\mathcal{H}_5(p)$ is the following
$$
\begin{pmatrix}
3+3\sqrt{2}&0&0&0&0\\
0&2+\sqrt{2}&1&0&0\\
0&0&2+\sqrt{2}&0&0\\
0&0&0&0&0\\
0&0&0&0&0\\
\end{pmatrix},
$$
so the double positive eigenvalue $2+\sqrt{2}$ has exactly one Jordan block that agrees with Theorem~\ref{Thm.eigenspaces}.

\setcounter{equation}{0}

%%%%%%%%%%%%%%%%%%%%%%%%%%%%%%%%%%%%%%%%%%%%%%%%%%%%%%%%%%%%%%%%%%%%%%%%%%%%%%%%%%%%%%%%%%%%%%%%%%%%%%%%%%%%%%%%%%%%%%%%%%%%%%%%%%%%%%%%%
\section{Conclusion}
%%%%%%%%%%%%%%%%%%%%%%%%%%%%%%%%%%%%%%%%%%%%%%%%%%%%%%%%%%%%%%%%%%%%%%%%%%%%%%%%%%%%%%%%%%%%%%%%%%%%%%%%%%%%%%%%%%%%%%%%%%%%%%%%%%%%%%%%%

In this paper, we have presented necessary and sufficient conditions on the containment of the zeros of a given real polynomial outside of a sector in the right-half plane of the complex plane for its finite Hurwitz matrix to be totally nonnegative. We have shown that the enclosing sectors are sharp. As a by-product, we have extended some results known from literature on the spectrum and the eigenspaces of the finite Hurwitz matrix and proved the conjecture posed in~\cite[Conjecture 3.49]{Holtz_Tyaglov}.

In closing, we mention an open problem related to our results. In~\cite{Goodman.Sun}, see also~\cite[Section 4.8]{Pinkus2010}, the \textit{generalized Hurwitz matrix}
\begin{equation}\label{generalized.Hurw.m}
H^{M}_{\infty}(p)=(a_{Mj-i})_{i,j=1}^{\infty}
\end{equation}
for a polynomial $p$ given in~\eqref{main.poly} and a natural number $M\leqslant n$ was introduced.
Here we use the convention that $a_k=0$ for $k < 0$ and $k > n$.

Note that for $M=1$, this matrix is just the matrix $\mathcal{T}(p)$ in~\eqref{Toeplitz.infinite.matrix}, so $\mathcal{T}(p)=H^1_{\infty}(p)$, and for $M=2$
the infinite Hurwitz matrix $H_{\infty}(p)$  in~\eqref{Hurwitz.matrix.infinite.for.poly}, so $H_{\infty}(p)=H^2_{\infty}(p)$.
In~\cite{HoltzKhrushKushel},
it was established that if $p$ has only positive coefficients and the matrix $H^M_{\infty}(p)$ is totally  nonnegative, then the polynomial
$p$ has no zeros in the sector $|\arg z|<\frac{\pi}M$. This result generalizes the results by Aissen-Edrei-Schoenberg-Whitey~\cite{AisenEdreiShoenbergWhitney}
$(M = 1)$, see Theorem 3.6, and the necessary condition in Theorem 4.10 $(M = 2)$, and is related to the theorem by
Cowling-Thron~\cite{Cowling.Thron} $(M = n)$. A challenging problem is to find suitable conditions on the location of the zeros of $p$ for the total nonnegativity of the
generalized Hurwitz matrix. Such a fact was established in~\cite[Theorems~2.1 and 2.4]{Goodman.Sun}: Suppose that $p$ has degree $n\leqslant(M-1)m+1$ for some integer $m\geqslant1$. If
$p$ has all its zeros in the sector
$$
|\pi-\arg z|<\dfrac{\pi}{m+1},
$$
then the matrix $H^{M}_{\infty}(p)$ is totally nonnegative\footnote{In fact, it is even almost strictly totally positive}.
However, if the degree of $p$ increases, this sector becomes smaller and smaller. Therefore, a sufficient condition which does not depend on the degree of $p$ is desired.
In~\cite{HoltzKhrushKushel}, it was shown that the condition that the polynomial~$p$ is
stable does not suffice for general $M$, $M\neq2$, (but is suffices if $M$ is even). Furthermore, nothing seems to be known about the eigenstructure of the
finite generalized Hurwitz matrices for general~$M$.

%%%%%%%%%%%%%%%%%%%%%%%%%%%%%%%%%%%%%%%%%%%%%%%%%%%%%%%%%%%%%%%%%%%%%%%%%%%%%%%%%%%%%%%%%%%%%%%%%%%%%%%%%%%%%%%%%%%%%%%%%%%%%%%%%%%%%%%%%
\section*{Acknowledgements}
%%%%%%%%%%%%%%%%%%%%%%%%%%%%%%%%%%%%%%%%%%%%%%%%%%%%%%%%%%%%%%%%%%%%%%%%%%%%%%%%%%%%%%%%%%%%%%%%%%%%%%%%%%%%%%%%%%%%%%%%%%%%%%%%%%%%%%%%%

The work of M.\,Adm leading to this publication was partially supported by the German Academic Exchange Service (DAAD) with funds from the German Federal Ministry of Education and Research (BMBF) and the People Programme (Marie Curie Actions) of the European Union’s Seventh Framework Programme (FP7/2007-2013) under REA grant agreement n° 605728 (P.R.I.M.E. – Postdoctoral Researchers International Mobility Experience).

The work of M.\,Tyaglov was partially supported by The Program for Professor of Special Appointment (Oriental Scholar) at Shanghai Institutions of
Higher Learning, by the Joint NSFC-ISF Research Program, jointly funded by the National Natural Science Foundation of China and the
Israel Science Foundation (No.11561141001), and by Grant AF0710021 from Shanghai Jiao Tong University University.

\setcounter{equation}{0}

%%%%%%%%%%%%%%%%%%%%%%%%%%%%%%%%%%%%%%%%%%%%%%%%%%%%%%%%%%%%%%%%%%%%%%%%%%%%%%%%%%%%%%%%%%%%%%%%%%%
\section{Appendix}
%%%%%%%%%%%%%%%%%%%%%%%%%%%%%%%%%%%%%%%%%%%%%%%%%%%%%%%%%%%%%%%%%%%%%%%%%%%%%%%%%%%%%%%%%%%%%%%%%%%

As announced in Section~\ref{section:stable.poly}, we give here a proof
of Proposition~\ref{Propos.Hurw.st.R.func} which seems to explicitly appear the first time in~\cite{Barkovsky.2} (and implicitly in~\cite{Hurwitz,Gantmakher}). The proof we present here is due to Yu.\,Barkovsky~\cite{Barkovsky.private},
and is based on properties of \textit{R}-functions. This proof seems to be new~\footnote{One more proof distinct from this one and from Gantmacher's proof
(see~\cite[Ch.~XV]{Gantmakher}) and more close to Hurwitz's proof~\cite{Hurwitz} was presented by Barkovsky~in~\cite{Barkovsky.2}.},
and we allow ourselves to reproduce it here.

In the sequel, we need the following property of \textit{R}-functions.
\begin{theorem}\label{Th.R-function.general.properties}
Let $h$ and $f$ be real polynomials such that $\deg
h-1\leqslant\deg f=n$. For the real rational function
\begin{equation*}\label{Rational.function.for.R-functions}
R=\dfrac{h}{f},
\end{equation*}
with exactly $r$ poles, the following conditions
are equivalent:
\begin{itemize}
\item[$1)$] $R$ is an \textit{R}-function:
\item[$2)$] The function $F$ can
be represented in the form
\begin{equation}\label{Mittag.Leffler.1}
\displaystyle R(z)=-\alpha z+\beta+\sum^{r}_{j=1}\frac
{\gamma_j}{z+\omega_j},\quad\alpha\geqslant0,\,\,\,
\beta,\omega_j\in\mathbb{R},
\end{equation}
where
\begin{equation}\label{Mittag.Leffler.1.poles.formula}
\gamma_j=\dfrac{h(\omega_j)}{f'(\omega_j)}>0,\quad j=1,\ldots,r.
\end{equation}

\end{itemize}
\end{theorem}
\noindent According to Definition~\ref{def.S-function}, \textit{R}-functions satisfy the condition~\eqref{R-function.negative.type.condition.doblicate}.

Now we provide some additional relations between the polynomials~\eqref{poly1.13}--\eqref{poly1.12} and the function~\eqref{assoc.function}.
Let the polynomial $p$ be given as in~\eqref{main.poly}. Then
the polynomials $p_0(z^2)$ and $p_1(z^2)$ satisfy the following
identities:
\begin{equation}\label{poly1.2}
\begin{split}
&\displaystyle p_0(z^2)=\frac{p(z)+p(-z)}2,\\
&\displaystyle p_1(z^2)=\frac {p(z)-p(-z)}{2z}.
\end{split}
\end{equation}

From~\eqref{poly1.2} and~\eqref{assoc.function} one can derive the
relation
\begin{equation}\label{poly1.1}
z\Phi(z^2)=\displaystyle\frac{p(z)-p(-z)}{p(z)+p(-z)}
=\displaystyle\frac{1-\displaystyle\frac{p(-z)}{p(z)}}{1+\displaystyle\frac{p(-z)}{p(z)}}.
\end{equation}

Let us recall the following simple necessary condition for
polynomials to be stable. This condition is usually called
the \textit{Stodola condition}~\cite{Gantmakher} (see also~\cite{Barkovsky.2}).

\begin{theorem}[Stodola]\label{Th.Stodola.necessary.condition.Hurwitz}
If the polynomial $p$ is stable, then all its coefficients
are positive\footnote{More precisely, the coefficients must be of
the same sign, but $a_0>0$ by~\eqref{main.poly}.}.
\end{theorem}

Now we are in a position to prove Proposition~\ref{Propos.Hurw.st.R.func}.

\vspace{2mm}

\noindent\textbf{Proposition~\ref{Propos.Hurw.st.R.func}.}(\cite{Gantmakher,Barkovsky.2})
\textit{The polynomial $p$ defined in~\eqref{main.poly} is
stable if and only if its associated function~$\Phi$
defined in~\eqref{assoc.function} is an \textit{R}-function with
exactly $l$ poles, all of which  are negative, and the limit
$\displaystyle\lim_{u\to\pm\infty}\Phi(u)$ is positive whenever
$n=2l+1$. The number $l$ is defined in~\eqref{floor.poly.degree}.}

%\vspace{1mm}

%
\begin{proof}
Let the polynomial $p$ be stable. %, that is,
%
%\begin{equation}\label{Theorem.main.Hurwitz.stability.proof.1}
%p(\lambda)=0\ \ \Longrightarrow\ \ \Re\lambda<0.
%\end{equation}
%
%
First, we show that
\begin{equation}\label{Theorem.main.Hurwitz.stability.proof.2.2}
\displaystyle\left|\frac {p(-z)}{p(z)}\right|<1,\quad\forall z:\,
\Re{z}>0.
\end{equation}
Note that the polynomials $p(z)$ and $p(-z)$ have no common zeros
if $p$ is stable, so the function $\dfrac {p(-z)}{p(z)}$
has exactly $n$ poles. The stable polynomial $p$ can be
represented in the form
\begin{equation*}\label{Theorem.main.Hurwitz.stability.proof.2.5}
p(z)=a_0\prod_k(z-\lambda_k)\prod_j\left(z-\xi_j\right)\left(z-\overline{\xi}_j\right),
\end{equation*}
where $\lambda_k<0,\Re\xi_j<0$ and $\Im\xi_j\neq0$. Then we have
\begin{equation}\label{poly1.5}
\displaystyle\left|\frac {p(-z)}{p(z)}\right|=\prod_{k}\frac
{\left|z+\lambda_k\right|}{\left|z-\lambda_k\right|}\prod_{j}\frac
{\left|z+\xi_j\right|\left|z+\overline{\xi}_j\right|}{\left|z-\overline{\xi}_j\right|\left|z-\xi_j\right|}.
\end{equation}
It is easy to see that the function of type
$\dfrac{z+a}{z-\overline{a}}$~, where $\Re a<0$, maps the right
half-plane to the~unit disk. In fact,
\begin{equation*}\label{Theorem.main.Hurwitz.stability.proof.3}
\left|\dfrac{z+a}{z-\overline{a}}\right|^2=\dfrac{(\Re z+\Re
a)^2+(\Im z+\Im a)^2}{(\Re z-\Re a)^2+(\Im z+\Im
a)^2}<1\qquad\text{whenever}\quad\Re z>0\quad\text{and}\quad\Re
a<0.
\end{equation*}
Now from~\eqref{poly1.5} it follows that the function
$\displaystyle\frac {p(-z)}{p(z)}$ also maps the right half-plane
to the~unit disk as a~product of functions of such a type. Thus,
the inequality~\eqref{Theorem.main.Hurwitz.stability.proof.2.2} is
valid.

At the same time, the fractional linear transformation
$\displaystyle z\mapsto\frac{1-z}{1+z}$ conformally maps the unit
disk to the right half-plane:
\begin{equation}\label{Theorem.main.Hurwitz.stability.proof.5}
|z|<1\implies\Re\left(\frac{1-z}{1+z}\right)=\dfrac{1-|z|^2}{|1+z|^2}>0.
\end{equation}
Consequently, from the
relations~\eqref{poly1.1},~\eqref{Theorem.main.Hurwitz.stability.proof.2.2},
and~\eqref{Theorem.main.Hurwitz.stability.proof.5} we obtain that
the function $z\Phi(z^2)$ maps the~right half-plane to itself, so
the function $-z\Phi(-z^2)$ maps the upper half-plane of the
complex plane to the lower half-plane:
\begin{equation*}\label{Theorem.main.Hurwitz.stability.proof.6}
\Im z>0\implies\Re(-iz)>0\implies
\Re\left[-iz\Phi(-z^2)\right]=\Im\left[z\Phi(-z^2)\right]>0\implies\Im\left[-z\Phi(-z^2)\right]<0.
\end{equation*}
Since $p$ is stable by assumption, the polynomials $p(z)$
and $p(-z)$ have no common zeros, therefore, $p_0$ and~$p_1$ also
have no common zeros, and $p_0(0)\neq0$
by~\eqref{app.poly.odd.even}. Moreover, by
Theorem~\ref{Th.Stodola.necessary.condition.Hurwitz} we have
$a_0>0$ and $a_1>0$, so $\deg p_0=l$
(see~\eqref{poly1.13} and~\eqref{poly1.12}). Thus, the number of
poles of the function $-z\Phi(-z^2)$ equals the number of the zeros
of the polynomial $p_0(-z^2)$, i.e., exactly $2l$.

So according to Theorem~\ref{Th.R-function.general.properties},
the function $-z\Phi(-z^2)$ can be represented in the
form~\eqref{Mittag.Leffler.1}, where all poles are located
symmetrically with respect to~$0$ and $\beta=0$, since
$-z\Phi(-z^2)$ is an odd function. Denote the poles of
$-z\Phi(-z^2)$ by $\pm\nu_1,\ldots,\pm\nu_l$ such that
\begin{equation*}\label{Th.stable.poly.and.R-function.proof.6.5}
0<\nu_1<\nu_2<\ldots<\nu_l.
\end{equation*}
Note that $\nu_1\neq0$, since $p_0(0)\neq0$ as we mentioned above.

Thus, the function $-z\Phi(-z^2)$ can be represented in the
following form
\begin{equation*}\label{Th.stable.poly.and.R-function.proof.7}
-z\Phi(-z^2)=-\alpha z+\sum_{j=1}^l\frac{\gamma_j}{z-\nu_j}+
\sum_{j=1}^l\frac{\gamma_j}{z+\nu_j}= -\alpha z
+\sum_{j=1}^l\frac{2\gamma_jz}{z^2-\nu_j^2},
\quad\alpha\geqslant0,\,\gamma_j,\nu_j>0.
\end{equation*}
By dividing this equality by $-z$ and changing variables as follows
$-z^2\to u$, $2\gamma_j\to\beta_j$, $\nu_j^2\to\omega_j$, we
obtain the following representation of the function $\Phi$:
\begin{equation}\label{Mittag.Leffler}
\Phi(u)=\dfrac{p_1(u)}{p_0(u)}=\alpha+\sum_{j=1}^l\frac{\beta_j}{u+\omega_j},
\end{equation}
where $\alpha\geqslant0,\,\beta_j>0$ and
\begin{equation*}\label{Th.stable.poly.and.R-function.proof.8}
0<\omega_1<\omega_2<\ldots<\omega_l.
\end{equation*}
Here $\alpha=0$ whenever $n=2l$, and
$\alpha=\displaystyle\frac{a_0}{a_1}>0$ whenever $n=2l+1$. Since
$\Phi$ can be represented in the form~\eqref{Mittag.Leffler}, we
have that by Theorem~\ref{Th.R-function.general.properties},
$\Phi$ is an \textit{R}-function with exactly $l$ poles, all of which
are negative, and
$\displaystyle\lim_{u\to\pm\infty}\Phi(u)=\alpha>0$ as $n=2l+1$.

\vspace{3mm}

Conversely, let the polynomial $p$ be defined as
in~\eqref{main.poly} and let its associated function $\Phi$
be an \textit{R}-function with exactly $l$ poles, all of which are
negative, and $\displaystyle\lim_{u\to\pm\infty}\Phi(u)>0$ as
$n=2l+1$. We will show that $p$ is stable.

By Theorem~\ref{Th.R-function.general.properties}, $\Phi$ can be
represented in the form~\eqref{Mittag.Leffler}, where
$\alpha=\displaystyle\lim_{u\to\pm\infty}\Phi(u)\geqslant0$ such
that $\alpha>0$ if $n=2l+1$, and $\alpha=0$ if $n=2l$. Thus, the
polynomial $p_0$ has only negative zeros, and the polynomials
$p_0$ and $p_1$ have no common zeros. Together
with~\eqref{app.poly.odd.even} and~\eqref{assoc.function}, this
implies that the set of zeros of the polynomial $p$ coincides
with the set of solutions of the equation
\begin{equation}\label{poly1.9}
z\Phi(z^2)=-1.
\end{equation}
Let $\lambda$ be a zero of the polynomial $p$ and therefore, a solution of
the equation~\eqref{poly1.9}. Then from~\eqref{Mittag.Leffler}
and~\eqref{poly1.9} we obtain
\begin{equation*}
-1=\Re\left[\lambda\Phi(\lambda^2)\right]=\left[\alpha+
\sum_{j=1}^l\beta_j\frac{|\lambda|^2+\omega_j}{|\lambda^2+\omega_j|^2}
\right]\Re\lambda,
\end{equation*}
where $\alpha\geqslant0$, and $\beta_j,\,\omega_j>0$ for
$j=1,\ldots,l$. Thus, if $\lambda$ is a zero of $p$, then
$\Re\lambda<0$, and so $p$ is stable.
\end{proof}

% ------------------------------------------------------------------------

%\addcontentsline{toc}{section}{Bibliography} %\dotfill}
%\bibliographystyle{plain}
%\bibliography{References}

\end{document}